\documentclass[12pt]{amsart}
\usepackage[utf8]{inputenc}
\usepackage{amsmath,amssymb,amsthm}
\usepackage{url}
\usepackage{hyperref}
\usepackage{tikz-cd}
\usepackage{stmaryrd}

\theoremstyle{definition}
\newtheorem{definition}{Definition}[section]

\newtheorem{observation}[definition]{Observation}
\newtheorem{question}[definition]{Question}
\newtheorem{fact}[definition]{Fact}
\newtheorem{notation}[definition]{Notation}
\newtheorem{claim}[definition]{Claim}

\theoremstyle{plain}
\newtheorem*{theorem*}{Theorem}

\theoremstyle{plain}
\newtheorem{theorem}[definition]{Theorem}
\newtheorem{proposition}[definition]{Proposition}
\newtheorem{lemma}[definition]{Lemma}
\newtheorem{corollary}[definition]{Corollary}

\newtheorem{remark}[definition]{Remark}

\usepackage{amsmath,amssymb,amsthm, xcolor, enumitem, comment, todonotes}

\title{Higher Dimensional Chain Conditions}
\author{Stevo Todorcevic and Jing Zhang}
\date{}

\address[Todorcevic]{\hfill\break Department of Mathematics, University of Toronto,
Canada
\hfill\break stevo@math.toronto.edu}

\address{\hfill\break Institut de Math\'ematiques de Jussieu, CNRS, Paris,
France
\hfill\break stevo.todorcevic@imj-prg.fr}

\address{\hfill\break Matemati\v{c}ki Institut, SANU, Belgrade, Serbia
\hfill\break stevo.todorcevic@sanu.ac.rs}

\address[Zhang]{\hfill\break Department of Mathematics,
University of Toronto
\hfill\break jingzhan@alumni.cmu.edu}

\thanks{
{\it 2010 \emph{MSC}.} 03E02, 03E10 \newline
{The research on this paper is partially supported by
grants from NSERC(455916),  CNRS(UMR7586) and SFRS(7750027-SMART)
}}

\begin{document}
\maketitle
\begin{abstract}
We investigate higher dimensional chain conditions, where the largeness notion is given by Fubini products of a given ideal. From strong saturation properties of an ideal, we derive abstractly versions of higher dimensional $\Delta$-system lemma, which imply many posets, including any finite support iteration of $\sigma$-centered posets and measure algebras, satisfy the higher dimensional chain conditions. We then show that if a poset satisfies a strengthening of the $\sigma$-finite chain condition by Horn and Tarski, then it satisfies higher dimensional chain conditions.  As an application, we derive Ramsey-theoretic consequences, namely various \emph{partition hypotheses} as studied by Bannister, Bergfalk, Moore and Todorcevic, from the existence of ideals satisfying strong chain conditions.
\end{abstract}

\section{Introduction and preliminaries}

Recall that a family $U$ of sets forms a \emph{$\Delta$-system with root $r$} if for any distinct $u,v\in U$, $u\cap v =r$. The most commonly known form of the $\Delta$-system lemma, due to Shanin \cite{Shanin}, is that when the sets in $U$ are finite: 
For every uncountable family $U$ of finite sets there exist an uncountable subcollection $U^*\subset U$ and $r$ such that $U^*$ forms a $\Delta$-system with root $r$. Shanin's $\Delta$-system lemma and its natural extension to sets of larger cardinality due to Erd\H{o}s and Rado \cite{ER} is a useful tool in proving that iterations and products of forcings have certain chain conditions. Thus, while the  original motivation behind the  notion of $\Delta$-system comes from topology (see \cite{CN}), it has become a standard tool in forcing ever since its famous  applications in this area by Cohen \cite{Cohen} and Solovay and Tennembaum \cite{ST}.

The 2-dimensional $\Delta$-system lemma first appeared in the work of Todorcevic \cite{reals} where it was used to analyze 2-dimensional partition relations in forcing extensions obtained by standard iterations and products of small forcings. 

It turns out that higher dimensional $\Delta$-system lemma has proven to be very useful in establishing consistency results regarding higher dimensional partition relations. Shelah \cite{SierpinskiI,SierpinskiII} developed the higher dimensional versions to give a consistent answer to a conjecture of Galvin. Some later applications and refinements of the higher dimensional $\Delta$-system lemma appear in \cite{MR398837, MR4275058, NatashaHL, MR2520110, MR4606258, MR4568267, MR4107536,MR3961609, MR4217964}.

One such recent application most relevant to this paper is the work of Bergfalk and Lambie-Hanson \cite{MR4275058} where they use the higher dimensional $\Delta$-system lemmas to analyze the combinatorics in the forcing extension by adding weakly compact many Hechler reals, which has the consequence that the $n$-th derived limit of a certain inverse system is $0$ for all $n\in \omega$. Later, the partition relations responsible for the algebraic consequences were extracted in \cite{bannister2023descriptive} and it was proved that these partition relations indeed hold in the weakly compact Hechler model.

A typical setup in proving these consistency results is that one starts with a cardinal in the ground satisfying strong partition relations. Then one analyzes the suitable forcing extensions utilizing these partition relations that hold in the ground model. As a result, in all of the applications above, the final model of set theory is a specific forcing extension over a suitable ground model.

The goal of this paper is to go one further step in this chain of reductions, deriving Ramsey-type consequences only from the existence of ideals with strong saturation properties, without making extra assumptions about the current universe. There is a long history of using saturated ideals (or more generally, compactness principles including large cardinals, forcing axioms and so on) to derive positive partition relations, for example, \cite{MR673792, MR1833480, MR1968607,MR716846, suslintrees}. In \cite{suslintrees}, the main result from \cite{reals} was extended by deriving the positive partition relations abstractly from the existence of ideals on the continuum with strong saturation properties, giving rise to many more consistent configurations. This work is a continuation of this line of research. To state the main result of this paper, we need a couple more definitions.

\begin{definition}
Let $I\subset P(\kappa)$ is an ideal on $\kappa$, let 
	\begin{enumerate}
	\item $I^*=\{A\subset \kappa: A^c\in I\}$ denote the dual filter on $\kappa$,
	\item $I^+=\{A\subset \kappa: A\not\in I\}$ denote the collection of positive sets.
	\end{enumerate}
 
We say $I$ is 
\begin{enumerate}
\item \emph{proper} if $\kappa\not\in I$,
\item  \emph{uniform} if $[\kappa]^{<\kappa}\subset I$,
\item \emph{$\eta$-complete} where $\eta$ is a cardinal if for any $\eta'<\eta$ and any $\langle A_i\in I: i<\eta'\rangle$, $\bigcup_{i<\eta'} A_i\in I$, and 
\item \emph{normal} if for any $\langle A_i\in I: i<\kappa\rangle$, $\bigtriangledown_{i<\kappa} A_i =_{def} \{\alpha<\kappa: \exists \beta<\alpha , \alpha\in A_\beta\}\in I$.
\end{enumerate} 

\end{definition}

\begin{definition}\label{definition: saturation}
An ideal $I$ on $\kappa$ is \emph{$\gamma$-saturated} if whenever $\langle X_i: i<\gamma\rangle\subset I^+$ are given, there are $i<j<\gamma$ such that $X_i\cap X_j\in I^+$.
\end{definition}

In this paper, unless otherwise specified, all ideals are assumed to be proper and uniform.

\begin{definition}
Let $I$ be an ideal on $\kappa$. Define the ideals $I^n$ on $\kappa^n$ for each $n\in \omega$ by recursion: $X\in I^n$ iff $\{\alpha: \{\bar{\beta}\in \kappa^{n-1}: (\alpha,\bar{\beta})\in X\}\in I^{n-1}\}\in I^*$.
\end{definition}

Note that if an ideal $I$ on $\kappa$ is $\kappa$-complete, then $I^n$ is $\kappa$-complete for any $n\in \omega$. Since all the ideals in this paper are uniform, we may assume $I^n$ concentrates on $[\kappa]^n =_{\mathrm{def}} \{a=\langle a_i: i<n\rangle\in \kappa^n: \forall i<n-1, a_i<a_{i+1}\}$.

\begin{definition}
For any $X\subset [\kappa]^n$, $k<n$ and $a\in [\kappa]^k$, we let 
\begin{enumerate}
\item $(X)_a $ denote the set $\{b\in [\kappa]^{n-k}: (a,b)\in X\}$, and 
\item $X_0 = \{\alpha\in \kappa: (X)_\alpha\in I^+\}$.
\end{enumerate}

\end{definition}

\begin{remark}
If $I$ is an ideal on $\kappa$, we will implicitly assume that if $X\in (I^n)^+$ and $(a,c)\in X$, then $(X)_a\in (I^{n-|a|})^+$. This is without loss of generality since sets satisfying this property are dense in $P([\kappa]^n)/I^n$.
\end{remark}

\begin{definition}
For $a,b\in [\kappa]^n$, we say $a\simeq b$ iff $(a,a\cap b, <) \simeq (b,a\cap b, <)$. Namely, the unique order preserving isomorphism from $a$ to $b$ is identity on $a\cap b$.
\end{definition}

\begin{definition}\label{definition: weakDelta}
Let $J$ be an ideal on $\kappa$.
We say that \emph{$\nu$-$J^n$-weak-$\Delta$-system lemma holds} if given any $X\in (J^n)^+$ and $\langle X_a\in [\kappa]^{<\nu}: a\in X\rangle$, there exist $W: [\kappa]^{\leq n-1} \to [\kappa]^{<\nu}$ and $J^n$-positive $Y\subset X$ such that for any $a\simeq b \in Y$, $X_a\cap X_b \subset W(a\cap b)$.
\end{definition}

\begin{remark}
In the case that $\nu\leq \kappa$, the assumption that each $X_a$ is a subset of $\kappa$ in Definition \ref{definition: weakDelta} is unnecessary. Namely, it is equivalent to the following: given any set $S$, $X\in (J^n)^+$ and $\langle X_a\in [S]^{<\nu}: a\in X\rangle$, there exist $W: [\kappa]^{\leq n-1} \to [S]^{<\nu}$ and $J^n$-positive $Y\subset X$ such that for any $a\simeq b \in Y$, $X_a\cap X_b \subset W(a\cap b)$.
\end{remark}

The first main result of this paper is an abstract derivation of weak $\Delta$-system lemmas from the existence of a saturated ideal.
 
\begin{theorem}\label{theorem: nDWeakDelta}
Let $J$ be a normal uniform $\gamma$-saturated ideal on $\kappa$ for some $\gamma<\kappa$. Then for $\nu \in [\gamma, \kappa)\cap cof(\geq \gamma)$, the $|\nu|$-$J^{n+1}$-weak-$\Delta$-system lemma holds for any $n\in \omega$.
\end{theorem}

We utilize Theorem \ref{theorem: nDWeakDelta} to prove that a large family of posets satisfy the \emph{higher dimensional chain conditions}, which are the key notions as we defined below.

\begin{definition}\label{definition: Knaster}
Fix $n\in \omega, m\in \omega\cup \{\omega\}$, a poset $P$ and an ideal $J$ on $\kappa$. We say 
\begin{enumerate}
\item that $P$ is $J^n$-$< m$-linked (or $J^n$-$m^-$-linked if $m=m^-+1$) if for any $X\in (J^n)^+$ and any $\langle p_a: a\in X\rangle$, there exists $Y\subset X$ in $(J^n)^+$ such that for any $\{a_i: i<k\}\subset Y$ with $k<m$ satisfying that for any $i,j<k$, $a_i\simeq a_j$, then there exists $q\in P$ such that $q\leq p_{a_i}$ for all $i<k$; 
\item that $P$ is $J^n$-Knaster if $P$ is $J^n$-2-linked.
\end{enumerate}
\end{definition}

\begin{theorem}\label{theorem: CohenRandom}
If $P$ is either 
	\begin{enumerate}
	\item a finite support iteration of $\sigma$-centered posets, or
	\item a measure algebra,
	\end{enumerate}
then for any uniform normal $\gamma$-saturated ideal $J$ on $\kappa$ with $\gamma<\kappa$, for any $n,m\in \omega$, $P$ is $J^n$-$m$-linked.
\end{theorem}

The $J^n$-$m$-linked-ness has implications on the validity of the partition relations $\mathrm{PH}_n(\kappa)$ as defined in \cite{bannister2023descriptive}, see Theorem \ref{theorem: PH}, which was the original motivation of our study. By the result of \cite{suslintrees}, if $\kappa$ carries a normal uniform ideal $J$ such that $P(\kappa)/J$ is $J^n$-Knaster for all $n\in\omega$, then $\kappa\to (\kappa, \alpha)_2^2$ for all $\alpha<\omega_1$.

Finally, we define a strengthening of the $\sigma$-finite chain condition which we name the $\sigma$-$w$-finite chain condition (see Definition \ref{definition: sigmawedgefinitecc}) studied by Horn and Tarski \cite{horntarski} and show that such strong 1-dimensional chain conditions have higher dimensional consequences. 

\begin{theorem}\label{theorem: n+1-D}
Suppose $J$ is a uniform normal $\nu$-saturated ideal on $\kappa$ for some $\nu<\kappa$ and $n\in \omega$.
If $P$ is $\sigma$-$w$-finite-c.c, then $P$ is $J^{n+1}$-Knaster.
\end{theorem}

The organization of the paper is: 
\begin{enumerate}
\item in Section \ref{section: 2d}, we provide a proof for the 2-dimensional case of Theorem \ref{theorem: nDWeakDelta}. Though logically weaker than the result of the next section, we choose such exposition for the ease of understanding as it contains key ideas with less notational density;
\item in Section \ref{section: n+1Knaster}, we prove full Theorem \ref{theorem: nDWeakDelta} building on the previous section;
\item in Section \ref{section: Ramsey}, we give a proof of Theorem \ref{theorem: CohenRandom} and Ramsey-theoretic applications;
\item in Section \ref{section: wedge}, we prove Theorem \ref{theorem: n+1-D}.
\item finally in Section \ref{section: openquestions}, we conclude with some open questions.
\end{enumerate}

%
%
%
%
%
%
%

We finish the introduction with one preliminary lemma and some standard facts about generic elementary embeddings. 

\begin{lemma}\label{lemma: highernormality}
If $J$ is a uniform normal ideal on $\kappa$, then for $n\in \omega$, $J^n$ is normal in the following sense: for any $\langle D_{a}\in (J^n)^*: a\in [\kappa]^{<\omega}\rangle$, $D=_{def}\Delta_{a\in [\kappa]^{<\omega}} D_a =_{def} \{b\in [\kappa]^n: \forall a\in [\kappa]^{<\omega}, a\subset \min b, b\in D_a\}\in (J^n)^*$.
\end{lemma}

\begin{proof}
Since $J$ is uniform and normal, $J$ is $\kappa$-complete. Suppose $D^c\in (J^n)^+$ for the sake of contradiction. We define the following function $f$ whose domain is $D^c$ so that for any $b\in D^c$, $f(b)$ returns the least $a\in [\min b]^{<\omega}$ such that $b\not\in D_a$. By the $\kappa$-completeness of $J$, we can find $J^n$-positive $X\subset D^c$ such that for any $b,b'\in X$ with $\min b = \min b'$, then $f(b)=f(b')$. Let $X_0=\{\alpha\in \kappa: \exists c\in [\kappa]^{n-1}, (\alpha,c)\in X\} \in J^+$. Let $g$ be a function on $X_0$ such that $g(\alpha)=f(\alpha,c)$ for some $c$ such that $(\alpha,c)\in X$. This is well-defined by the preliminary pruning. By the normality of $J$, we can find $J$-positive $X_0'\subset X_0$ and some finite $a\in [\kappa]^{<\omega}$ such that $g\restriction X_0' \equiv \{a\}$. Let $Y=\{(\alpha,c): \alpha\in X_0', c\in (X)_\alpha\}$. Then $Y\in (J^n)^+$ and $f'' Y \equiv \{a\}$. By the definition of $f$, this means $Y\cap D_a=\emptyset$. This contradicts with the fact that $Y\in (J^n)^+$ and $D_a\in (J^n)^*$.
\end{proof}

The proof of the following facts can be found in \cite[Chapter 2]{ForemanChapter} and \cite[Chapter 16,17]{Kanamori}

\begin{fact}\label{fact: standard}
Let $\gamma\leq \nu<\kappa$ be cardinals.
Suppose $J$ is a $\nu$-saturated normal uniform ideal on $\kappa$ and $P$ is a $\gamma$-c.c forcing. Then in $V[G]$ where $G\subset P$ is generic over $V$: 
	\begin{enumerate}
	\item the ideal $\bar{J}$ generated by $J$ is uniform and normal;
	\item for any $n\in \omega$, $\bar{J}^n$ is generated by $J^n$, moreover, for any $p\in P$ and a $P$-name $\dot{X}$, such that $p\Vdash \dot{X}\in \bar{J}^n$, there exists $X^*\in J^n$ such that $p\Vdash \dot{X}\subset X^*$;
	\item $\bar{J}$ is $\nu$-saturated in $V[G]$ (see \cite[Theorem 17.1]{Kanamori});
	\item in a further forcing extension by $P(\kappa)/\bar{J}$, say $V[G*H]$, there exist a transitive class $M$ and an elementary embedding $j: V[G]\to M$ such that $\mathrm{crit}(j)=\kappa$ and $V[G*H]\models {}^{\kappa}M \subset M$;
	\end{enumerate}
\end{fact}

\section{$J^2$-weak-$\Delta$-system lemma}\label{section: 2d}

\begin{theorem}\label{theorem: 2DWeakDelta}
Let $J$ be a normal uniform $\gamma$-saturated ideal on $\kappa$ for some $\gamma<\kappa$. Then for $\nu\in [\gamma,\kappa)\cap cof(\geq \gamma)$, the $|\nu|$-$J^2$-weak-$\Delta$-system lemma holds.
\end{theorem}

Fix the ideal $J$ and $\nu\in [\gamma, \kappa)\cap cof(\geq \gamma)$ as in the hypothesis of Theorem \ref{theorem: 2DWeakDelta}.

\begin{lemma}\label{lemma: 1DShrinking}
For any $X\in J^+$ and $\vec{E}=\langle E_\alpha\in [\kappa]^{<\nu}: \alpha\in X\rangle$, there exist $Y\in [\kappa]^{<\nu}$ and $C\in J^*$ such that for any $\alpha<\beta\in X\cap C$, $\sup E_\alpha < \beta$ and $E_\alpha\cap \alpha \subset Y$, in particular, $E_\alpha\cap E_\beta \subset Y$.
\end{lemma}

\begin{proof}
Let $G\subset P(\kappa)/J$ be generic over $V$ containing $X$. Let $j: V\to M\simeq \mathrm{Ult}(V,G)$ be the generic ultrapower embedding. 
Let $E^*= j(\vec{E})(\kappa)$. Then $E^*\cap \kappa$ is a $<\nu$-subset of $\kappa$. Since $P(\kappa)/J$ satisfies $\gamma$-c.c and $\gamma\leq cf(\nu)$, there is some $Y\in V \cap [\kappa]^{<\nu}$ such that $\Vdash_{P(\kappa)/J} \dot{E}^*\cap \kappa\subset Y$. By the elementarity, we know that $\{\alpha<\kappa: E_\alpha\cap \alpha \subset Y\}\in G$. Since $G$ is chosen arbitrarily containing $X$, we know there must exist some $C'\in J^*$ such that $C'\cap X\subset \{\alpha<\kappa: E_\alpha\cap \alpha \subset Y\}$. Let $C=\{\alpha\in C': \forall \beta<\alpha, \sup X_\beta <\alpha\}$ and it is easy to see that $C\in J^*$, by the normality of $J$, is as desired.
\end{proof}

Let us fix $X\in (J^2)^+$ and $\langle E_a \in [\kappa]^{<\nu}: a\in X\rangle$. Applying Lemma \ref{lemma: 1DShrinking} repeatedly on both coordinates to shrink $X$ if necessary, we may assume there is $W: [\kappa]^{\leq 1}\to [\kappa]^{<\nu}$ such that for any $a,b\in X$, 
	\begin{enumerate}
	\item if $a_0=b_0$ and $a_1<b_1$, then $E_a\cap a_1\subset W(a_0)$, $E_b\cap b_1\subset W(b_0)$ and $\sup E_a < b_1$. In particular, $E_a\cap E_b\subset W(a\cap b)$.
	\item for $\alpha<\beta \in X_0$, $W(\alpha)\cap \alpha\subset W(\emptyset)$, $W(\beta)\cap \beta\subset W(\emptyset)$ and $\sup W(\alpha)<\beta$. In particular, $W(\alpha)\cap W(\beta)\subset W(\emptyset)$.
	\end{enumerate}
	Without loss of generality, we may assume any $B\in range(W)$ is of limit order type.	All that is left is to take care of pairs of the following types: 
	\begin{enumerate}
	\item $a,b$ such that $a\cap b =\emptyset$,
	\item $a,b$ such that $a_1 = b_1$.
	\end{enumerate}

\begin{proposition}\label{proposition:2-DNonOverlapping}
There exists a $J^2$-positive $Y\subset X$ such that for any $a,b\in Y$ with $a\cap b =\emptyset$, we have that $E_a \cap E_b \subset W(\emptyset)$.
\end{proposition}

\begin{proof}
Note that $\{\beta\in X_0: \forall a<\beta, \sup E_a <\beta\}$ is in $J^*\restriction X_0$. As a result, we may without loss of generality assume that for any $a<b\in X$, $E_a\cap E_b \subset W(\emptyset)$. 

Observe that for any $H\in [\kappa]^{<\nu}$ and any $\alpha\in X_0$, the collection $\{\beta\in (X)_\alpha: H\cap E_{\{\alpha,\beta\}}\subset W(\alpha)\}\in J^*\restriction (X)_\alpha$. By the normality of $J$, we define $J^2$-positive $Y\subset X$ such that for each $\alpha\in X_0=Y_0$, $$(Y)_\alpha=_{def}$$$$\{\beta\in (X)_\alpha: \forall a<\beta, \sup E_a < \beta,$$$$ \text{if $\min (E_{\{\alpha,\beta\}}-W(\alpha))\neq \emptyset$, then } \min (E_{\{\alpha,\beta\}}-W(\alpha))\geq\sup \bigcup_{\gamma<\beta} W(\gamma)\}$$$$\in J^*\restriction (X)_\alpha.$$
Let us check that $Y$ works. It is only left to verify $a=\{\alpha_0, \beta_0\}$ and $b=\{\alpha_1,\beta_1\}$ such that $\beta_0>\alpha_1$ or $\beta_1>\alpha_0$. For definiteness, assume $\beta_1>\alpha_0$. The proof for the other case is similar. The possible configurations are $\alpha_1<\alpha_0<\beta_0<\beta_1$ or $\alpha_0<\alpha_1<\beta_0<\beta_1$. In either case, we have that $E_a\cap E_b \subset E_b\cap \beta_1\subset W(\alpha_1)$. On the other hand, $\min (E_a - W(\alpha_0))\geq \sup W(\alpha_1)$. As a result, $E_a\cap E_b\subset W(\alpha_0)\cap W(\alpha_1)\subset W(\emptyset)$.\end{proof}

\begin{proposition}\label{proposition:2-DOverlapping}
There exist a $J^2$-positive $Y\subset X$ and $g: \kappa\to [\kappa]^{<\nu}$ such that for any $a,b\in Y$ with $a_1=b_1$, we have that $E_a \cap E_b \subset g(a\cap b)$.
\end{proposition}

\begin{proof}
First we claim that there is a condition $A\in P(\kappa)/J$ forcing that $B=\{\alpha: \kappa\in j((X)_\alpha)\}\in \bar{J}^+$. Suppose not for the sake of contradiction. By the fact that $P(\kappa)/J$ is $\gamma$-c.c, we can find $\langle B_i\in J: i<\gamma'\rangle$ for some $\gamma'<\gamma$ such that $\Vdash_{P(\kappa)/J} B\subset \bigcup_{i<\gamma'} B_i$. Since $J$ is $\kappa$-complete, we know $B^*=\bigcup_{i<\gamma'} B_i\in I$. Since $X_0\in J^+$, we can find $\alpha\in X_0\cap (B^*)^c$. However, the fact that $\alpha\in X_0$ implies that $(X)_\alpha\Vdash \alpha\in B$ and the fact that $\alpha\in (B^*)^c$ implies that $\Vdash \alpha\not\in B$, which is a contradiction.

Find some generic extension $G\subset P(\kappa)/J$ containing $A$ so that $B\in \bar{J}^+$, where $j: V\to M\simeq \mathrm{Ult}(V,G)$ is the generic elementary embedding satisfying that 
\begin{itemize}
\item $M$ is a transitive class,
\item $V[G]\models {}^\kappa M\subset M$.
\end{itemize}
Such generic extension exists by the $\gamma$-c.c-ness of $P(\kappa)/J$ (see \cite[Propositions 2.9 and 2.14]{ForemanChapter} for more details).

By Fact \ref{fact: standard}, $\bar{J}$ is $\gamma$-saturated in $V[G]$. Apply Lemma \ref{lemma: 1DShrinking} to $\{j(\vec{E})_{\alpha,\kappa}=_{def} E^*_{\alpha,\kappa}: \alpha\in B\}$ and $\bar{J}$ in $V[G]$, we can find some $W\in [j(\kappa)]^{<\nu}$ and $C\in J^*$ such that for all $\alpha, \beta\in C\cap B$, $E^*_{\alpha,\kappa}\cap E^*_{\beta,\kappa}\subset W$. Since $V[G]\models {}^\kappa M\subset M$, all of the sets found above are members of $M$. In particular, there is some $g\in V$ with $g: \kappa\to [\kappa]^{<\nu}$ such that $W=j(g)(\kappa)$. For each $\alpha\in C\cap X_0$, find some $J$-positive $Z_\alpha\subset (X)_\alpha$ forcing that $\alpha\in \dot{B}$ if we can. Otherwise $Z_\alpha$ is undefined.

	\begin{claim}
	$Z'=_{def}\{\alpha\in X_0: Z_\alpha \text{ is defined}\}\in J^+$.
	\end{claim}
	\begin{proof}[Proof of the Claim]
	Otherwise, there is $D\in J^*$ such that for any $\alpha\in D\cap X_0$, $(X)_\alpha\Vdash \alpha\not \in \dot{B}$. In $V[G]$, find $\alpha\in D\cap B$. Since $\kappa\in j((X)_\alpha)$, we have $(X)_\alpha\in G$, hence $\alpha\not\in B$, which is a contradiction.
	\end{proof}
	Finally, let $Z\subset X$ be the $J^2$-positive set $\{(\alpha,\alpha'): \alpha\in Z'\cap C, \alpha'\in Z_\alpha\}$. For any $\alpha<\beta\in Z_0$, if $(Z)_\alpha\cap (Z)_\beta$ is in $J^+$, then it forces that $\alpha,\beta\in C\cap \dot{B}$, in particular, $E^*_{\alpha,\kappa}\cap E^*_{\beta,\kappa}\subset j(g)(\kappa)$. As a result, there exists $C_{\alpha,\beta}\in J^*$ such that either $C_{\alpha,\beta}\cap (Z)_\alpha\cap (Z)_\beta =\emptyset$ or for any $\eta\in C_{\alpha,\beta}\cap (Z)_\alpha\cap (Z)_\beta \in J^+$, it is the case that $E_{\alpha, \eta}\cap E_{\beta, \eta}\subset g(\eta)$. Finally, let $Y=\{(\alpha,\alpha')\in Z: \alpha'\in \bigcap_{\delta<\alpha}C_{\delta,\alpha}\}$. It is easy to see $Y$ and $g$ are as desired.

\end{proof}

\begin{proof}[Proof of Theorem \ref{theorem: 2DWeakDelta}]
Apply Lemma \ref{lemma: 1DShrinking}, Propositions \ref{proposition:2-DNonOverlapping} and \ref{proposition:2-DOverlapping} successively. The final function as desired will be $W^*: [\kappa]^{\leq 1} \to [\kappa]^{<\nu}$ such that $W^*(\emptyset)=W(\emptyset)$ and $W^*(\beta)=W(\beta) \cup g(\beta)$, where $W$ and $g$ are as produced in the proofs above.
\end{proof}

\section{$J^{n+1}$-weak-$\Delta$-system lemma}\label{section: n+1Knaster}
We prove the general case in this section, by induction on $n\in \omega$. We restate Theorem \ref{theorem: nDWeakDelta} for the convenience of the reader.
\begin{theorem*}
Let $J$ be a normal uniform $\gamma$-saturated ideal on $\kappa$ for some $\gamma<\kappa$. Then for $\nu\in [\gamma,\kappa)\cap cof(\geq \gamma)$, the $|\nu|$-$J^{n+1}$-weak-$\Delta$-system lemma holds where $n\in \omega$.
\end{theorem*}
We may assume $n\geq 2$ and $\nu\geq \gamma$. Fix $X\in (J^{n+1})^+$ and $\langle E_{a}\in [\kappa]^{<\nu}: a\in X\rangle$ as in the hypothesis.

\subsection{Preliminary pruning}

For $i\leq n+1$, define $X\restriction i = \{a\in [\kappa]^i: \exists c\in [\kappa]^{n+1-i} (a,c)\in X\}$. It is clear that for all $i\leq n+1$, $X\restriction i \in (J^i)^+$. For each $a\in X\restriction i$ and $j> i$, let $(X\restriction j)_a = \{c\in [\kappa]^{j-i}: (a,c)\in X\restriction j\}$.

\begin{notation}
For two finite sequences $a$ and $b$, we use $a\sqsubseteq b$ to denote that either $a=b$ or $a$ is an initial segment of $b$.
\end{notation}

\begin{lemma}\label{lemma: rootsDelta}
There exists a $J^{n+1}$-positive $Y\subset X$ and $\langle F_{a}\in [\kappa]^{<\nu}: a\in Y\restriction i, i\leq n+1\rangle$ such that 
\begin{enumerate}
	\item if $a\in Y$, then $E_a \subseteq F_a$,
	\item for any $i\leq n$, $a\in Y\restriction i$ and for any $\eta_0<\eta_1\in (Y\restriction i+1)_{a}$, $\sup F_{a^\frown \eta_0}<\eta_1$ and $F_{a^\frown \eta_i}\cap \eta_i \subset F_a$ for $i=0,1$, in particular, $F_{a^\frown \eta_0} \cap F_{a^\frown \eta_1}\subset F_a$,
	\item for any $a'\sqsubseteq a$, $F_{a'}\subseteq F_{a}$.
\end{enumerate}
\end{lemma}

\begin{proof}
We induct on $n\geq 0$. When $n=0$, Lemma \ref{lemma: 1DShrinking} gives the desired conclusion. Suppose $n>0$. For each $a\in X\restriction n$, apply Lemma \ref{lemma: 1DShrinking} to get $F_a\in [\kappa]^{<\nu}$ and $J$-positive $Z_a\subset (X)_a$ such that for any $\eta'<\eta\in Z_a$, $F_{a^\frown \eta}\cap \eta \subset F_a$ and $\sup F_{a^\frown \eta'}<\eta$.

Apply the induction hypothesis to $X\restriction n$ and $\langle F_a: a\in X\restriction n\rangle$. We get a $J^n$-positive $Y^*\subset X\restriction n$ and $\langle F_a: a\in Y^*\restriction i, i\leq n\rangle$ satisfying the conclusion of the lemma. 

So far we have obtained the following sets $Y=\{(a,\gamma): a\in Y^*, \gamma\in Z_a\}$ and $\langle F_a: a\in Y^*\restriction i, i\leq n\rangle \cup \langle F_a=_{def}E_a: a\in Y\rangle$ satisfying the first two requirements. In order to satisfy the third requirement, let us modify the sets by thickening $F_a$'s and thinning out $Y$. Let $F^*_{a}=\bigcup_{a'\sqsubseteq a} F_{a'}$. Utilizing the normality of $J$ to shrink $Y$ if necessary, we can without loss of generality assume that whenever $a\in Y\restriction n$ and $\alpha'\in (Y)_a$, it is the case that for all  $b\in [\alpha']^{n+1}$ with $b\in Y$, we have $\sup F_{b}+1 < \alpha'$.
We check the second requirement is still maintained: for any $a$, $\eta_0<\eta_1$ as given, $\sup F^*_{a^\frown \eta_0}=\sup_{a'\sqsubseteq a} F_{a'}<\eta_1$ by the choice of $Y$ and $F^*_{a^\frown \eta_i}\cap \eta_i = \bigcup_{a'\sqsubseteq a^\frown \eta_i} F_{a'}\cap \eta_i \subseteq \bigcup_{a'\sqsubseteq a} F_{a'} \cup (F_{a^\frown \eta_i}\cap \eta_i)\subseteq F^*_a \cup F_a = F^*_a$.\end{proof}

\subsection{Non-overlapping increasing types}
Applying Lemma \ref{lemma: rootsDelta} to shrink $X$ if necessary, we may without loss of generality assume that there is $\langle F_a: a\in X\restriction i: i\leq n+1\rangle$ satisfying the conclusion of Lemma \ref{lemma: rootsDelta}.

The following is an easy consequence of Lemma \ref{lemma: rootsDelta}. 
\begin{lemma}\label{lemma:n+1increasingdisjoint}
There exist a $J^{n+1}$-positive $Y\subset X$ and $r\in [\kappa]^{<\nu}$ such that for any $a<b\in Y$, $F_a\cap F_b \subset r$.
\end{lemma}

The key point is that one can repeatedly apply (2) of Lemma \ref{lemma: rootsDelta} to yield $F_{a}\cap \min a\subset F_\emptyset$. Hence, in fact, we could take $r$ to be $F_\emptyset$.

\subsection{Overlapping at the last coordinate}

\begin{lemma}\label{lemma:lastoverlapDelta}
There exist a $J^{n+1}$-positive $Y\subset X$ and $g: [\kappa]^{\leq n} \to [\kappa]^{<\nu}$ such that for any $a\simeq b\in Y$ with $a_n=b_n$, $F_a\cap F_b\subset g(a\cap b)$.
\end{lemma}

\begin{proof}

We first show that $\{a\in X\restriction n: (X)_a\Vdash \dot{B}\in (\bar{J})^+\}\in (J^n)^+$, where $B=\{a\in X\restriction n: \kappa\in j((X)_a)\}$ is defined in $V^{P(\kappa)/J}$ and $j$ is the generic ultrapower defined by the generic ultrafilter extending $J$ over $V$. Suppose otherwise for the sake of contradiction, there is $E\in (J^n)^*$ such that for any $a\in E\cap X\restriction n$, $(X)_a \Vdash \dot{B}\in \bar{J}$. Since $P(\kappa)/J$ satisfies $\gamma$-c.c and $J$ is $\kappa$-complete, we can find $G_a\in J$ such that $(X)_a\Vdash \dot{B}\subset G_a$ for each $a\in E\cap X\restriction n$. Since $X$ is $J^{n+1}$-positive, we can find $(a, \beta)\in X\cap E^*$ where $E^*=_{def}\{(a_0,a_1): a_0\in E, a_1\in G_a^c\}\in (J^{n+1})^*$. On one hand, the fact that $(a,\beta)\in X$ implies that $(X)_a\Vdash a\in \dot{B}$. On the other hand, the fact that $(a,\beta)\in E^*$ implies that $(X)_a\Vdash a\not\in \dot{B}$. This is a contradiction. 

By shrinking $X$ if necessary, we may assume without loss of generality that for any $a\in X\restriction n$, $(X)_a\Vdash \dot{B}\in (\bar{J})^+$.
Let $G\subset P(\kappa)/J$ be generic over $V$ such that $B \in (\bar{J}^n)^+$. In $V[G]$, fix a generic elementary embedding $j: V\to M\simeq \mathrm{Ult}(V,G)$ with $\mathrm{crit}(j)=\kappa$ and $V[G]\models {}^\kappa M\subset M$.

Consider $D=\{j(\vec{F})_{a^\frown \kappa}=_{def}F^*_{a^\frown \kappa}: a\in B\}$. By Fact \ref{fact: standard}, $\bar{J}$ is $\gamma$-c.c in $V[G]$. Apply the induction hypothesis, more precisely Theorem \ref{theorem: nDWeakDelta} for dimension $n-1$, in $V[G]$ to $D$, $B$ and $\bar{J}$, we get $B^*\subset B$ which is $\bar{J}^n$-positive and a function $L$ such that for any $c,d\in B^*$, if $c\simeq d$, then $F^*_{a^\frown \kappa}\cap F^*_{b^\frown \kappa}\subset L(a\cap b)$. Since $V[G]\models {}^\kappa M\subset M$, we know that $L\in M$. As a result, there is some $l\in V$ such that $j(l)(\kappa)=L$. We may assume that for each $\gamma<\kappa$, $l(\gamma)$ is a function on $[\gamma]^{\leq n-1}$ to $[\kappa]^{<\nu}$.

Back in $V$. For each $a\in X\restriction n$, find $Z_a\subset (X)_a\in J^+$ forcing over $P(\kappa)/J$ that $a\in \dot{B}^*$ if such a set exists. Note that each such $(X)_a$ forces $\dot{B}\in (\bar{J})^+$, so we have no problem defining $\dot{B}^*$.

	\begin{claim}
	$Y^*=\{a\in X\restriction n: Z_a\text{ exists}\}\in (J^n)^+$.
	\end{claim}
	\begin{proof}[Proof of the Claim]
	Suppose otherwise for the sake of contradiction. There exists $C\in (J^n)^*\restriction (X\restriction n)$ such that for any $a\in C$, $(X)_a\Vdash a\not\in \dot{B}^*$. In $V[G]$, since $B^*\subset X\restriction n$ is in $(\bar{J}^n)^+$, we know that $B^*\cap C\neq \emptyset$. Fix $d\in B^*\cap C$. In particular, since $d\in B$, $\kappa\in j((X)_d)$, which implies that $(X)_d\in G$. As $(X)_d\Vdash d\not\in \dot{B}^*$, we have $d\not\in B^*$. This is a contradiction.
	\end{proof}
	
	For any $a\simeq b\in Y^*$, find $C_{a,b}\in J^*$ such that 
	\begin{itemize}
	\item either $C_{a,b}\cap Z_a\cap Z_b=\emptyset$ or 
	\item if $Z_a\cap Z_b\in J^+$, then for any $\gamma\in C_{a,b}\cap Z_a\cap Z_b$, $F_{a^\frown \gamma} \cap F_{b^\frown \gamma}\subset l(\gamma)(a\cap b)$. 
 \end{itemize}		
	The reason we can find $C_{a,b}$ is that $Z_a\cap Z_b \Vdash a, b\in \dot{B}^*$. 
	
	For each $b\in Y^*$, let $$Z_b'=\{\eta\in Z_b: \forall a<\eta, a\in Y^*, a\simeq b, \eta\in C_{a,b}\}.$$
	
	Note that $Z_b'=_J Z_b$ since $J$ is normal. Finally, let $Y=\{(b,\eta): b\in Y^*, \eta\in Z_b'\}$. It is easy to see that $Y\subset X$ is in $(J^{n+1})^+$. Let $g: [\kappa]^{\leq n} \to [\kappa]^{<\nu}$ defined such that $g(a)=l(\max(a))(a-\{\max (a)\})$.
	To see that $Y$ and $g$ are as desired, fix $a\simeq b\in Y$ with $a_n=b_n=\eta$. Let $a_0=a-\{\eta\}$ and $b_0=b-\{\eta\}$. By the definition of $Z_{b_0}'$ and $Z_{a_0}'$, $\eta\in C_{a_0,b_0}\cap Z_{a_0}\cap Z_{b_0}$. In particular, $Z_{a_0}\cap Z_{b_0}\in J^+$. As a result of the choice of $C_{a_0,b_0}$ in this case, $F_{a_0^\frown \eta} \cap F_{b_0^\frown \eta}\subset l(\eta)(a_0\cap b_0) = g(a\cap b)$.	\end{proof}

\subsection{Other types}

We classify the remaining types as follows. 

\begin{definition}\label{definition: typek}
We say $(a,b)\in [X]^2_{\simeq} =\{(c,d)\in X\times X: c_n\leq d_n, c\simeq d\}$ is of \emph{type $k$} for $0<k\leq n$ if $|b-a_n|=k$.
\end{definition}

\begin{remark}
The reason why we focus on $k\in (0,n+1)$ is that $(a,b)$ of type $0$ implies $a_n=b_n$ and $(a,b)$ of type $n+1$ implies $a<b$. Both of these types have been dealt with in the previous subsections.
\end{remark}

\begin{proposition}\label{proposition: typekDelta}
Fix $k\in (0,n+1)$. There exist a $J^{n+1}$-positive $Y\subset X$ and $l_k: [\kappa]^{<\omega}\to [\kappa]^{<\nu}$ such that for any $(a,b)\in [Y]^2_{\simeq}$ of type $k$, it is true that $F_a\cap F_b\subset l_k(a\cap b)$.
\end{proposition}

We remark that the definition of $l_k$ does not depend on $k$ and only depends on the choice of $Y$. I choose the indexing to emphasize that we are dealing with tuples of type $k$.

\begin{proof}
Recall that $\langle F_a: a\in X\restriction i, i\leq n+1\rangle$ satisfies the conclusion of Lemma \ref{lemma: rootsDelta}. Apply the induction hypothesis to $\langle F_a: a\in X\restriction n\rangle$ to get $J^n$-positive $Z\subset X\restriction n$ and $l_k$ such that for any $c,d\in Z$ with $c\simeq d$, $F_c\cap F_d\subset l_k(c\cap d)$. 
We say $(a^0, {b^0})$ is \emph{potentially of type $k$} if 
	\begin{itemize}
	\item $a^0\in Z$,
	\item ${b^0}\in Z\restriction n-k$,
	\item there are ${a^1}, {b^1}$ such that $({a^0}^\frown {a^1}, {b^0}^\frown {b^1})\in [X]^2_{\simeq}$ of type $k$.
	\end{itemize}

	\begin{claim}\label{claim: assym}
	For any $(a^0,{b^0})$ potentially of type $k$, there are $C_{a^0}=C_{a^0,{b^0}}^0\in J^*$ and $C_{{b^0}}=C^1_{a^0,{b^0}}\in (J^k)^*$ such that for any ${a^1}\in C_{a^0}\cap (X)_{a^0}$ and ${b^1}\in C_{{b^0}}\cap (X)_{{b^0}}$ with ${a^1}<{b^1}$, $F_{{a^0}^\frown {a^1}} \cap  F_{{b^0}^\frown {b^1}} \subset l_k(a^0\cap b^0)$.
	\end{claim}
	\begin{proof}[Proof of the claim]
	Let $C_{a^0}$ consist of $\alpha$ such that for all $b'<\alpha$, $\sup F_{b'} <\alpha$. Similarly $C_{b^0}$ consists of $c$ such that for all $b'<\min c$, $\sup F_{b'}<\min c$. Fix $a^1<b^1$ with $a^1\in C_{a^0}$ and $b^1\in C_{b^0}$. In particular, $a^1>\sup F_{b^0}$ and $\min b^1>\sup F_{{a^0}^\frown a^1}$.
	 By Lemma \ref{lemma: rootsDelta}, we know that $F_{{b^0}^\frown {b^1}}\cap \min b^1 \subset F_{b^0}$ and $F_{{a^0}^\frown a^1}\cap a^1 \subset F_{a^0}$. As a result, $F_{{b^0}^\frown b^1}\cap F_{{a^0}^\frown a^1}\subset F_{b^0}\cap F_{a^0}\subset F_{{b^0}^\frown b^1\restriction k-1}\cap F_{a^0} \subset l_k(a^0\cap b^0)$.
	\end{proof}
	
For each $b\in Z\restriction n-k$, let $$W_b = \{{b^1}: \forall a^0\subset \min {b^1}, (a^0,b) \text{ is potentially of type $k$} \rightarrow {b^1}\in C^1_{a^0, b} \}.$$
By Lemma \ref{lemma: highernormality}, $W_b\in (J^k)^*$. Let $$Y'=\{(c,d): c\in Z\restriction n-k, d\in W_c\cap (Z)_c\}.$$
Then $Y'$ is a $J^{n+1}$-positive subset of $X$.
For each $a\in Y'\restriction n$, let
$$E_a=\{\eta: \forall {b^0}\subset \eta, (a,{b^0}) \text{ is potentially of type $k$}\rightarrow \eta\in C^0_{a,{b^0}}\}.$$
By Lemma \ref{lemma: highernormality}, $E_a\in J^*$. Let $Y=\{(a,\eta)\in Y': \eta\in E_a\cap (Y')_a\}$. It is immediate that $Y$ is a $J^{n+1}$-positive subset of $Y'$.

Finally, let us verify that $Y$ is as desired. Suppose $(a,b)\in [Y]^2_{\simeq}$ of type $k$. Let $a^0=a\restriction n$ and ${b^0}=b\restriction n-k$. Then $(a^0, {b^0})$ is potentially of type $k$. Since ${b^0}<a_n$, $a_n\in C^{0}_{a^0, {b^0}}$ and since $a^0< b_{n-k+1}$, $b\restriction_{[n-k+1, n]}\in C^1_{a^0, {b^0}}$. Since $a_n < b\restriction_{[n-k+1, n]}$, by Claim \ref{claim: assym}, $F_{a}\cap F_{b}\subset l_k(a^0\cap b^0)= l_k(a\cap b)$.
\end{proof}

By going through Proposition \ref{proposition: typekDelta} for each value $k\in (0,n+1)$, we get the following.

\begin{corollary}\label{cor: alltypes1}
There exist a $J^{n+1}$-positive $Y\subset X$ and $l^*: [\kappa]^{<\omega}\to [\kappa]^{<\nu}$ such that for any $k\in (0,n+1)$, and any $(a,b)\in [Y]^2_{\simeq}$ of type $k$, it is true that $F_a\cap F_b\subset l^*(a\cap b)$.
\end{corollary}

\begin{proof}[Proof of Theorem \ref{theorem: nDWeakDelta}]
Apply successively Lemma \ref{lemma:n+1increasingdisjoint}, Lemma \ref{lemma:lastoverlapDelta} and Corollary \ref{cor: alltypes1}.
\end{proof}

\section{Applications to Ramsey-type results}\label{section: Ramsey}

In this section, we give some examples of $J^n$-$m$-linked posets.

\begin{theorem}\label{theorem: Cohen}
For any cardinal $\lambda$, any $\gamma$-saturated normal uniform ideal $J$ on $\kappa$ where $\gamma<\kappa$ and $m,n\in \omega$, $Add(\omega, \lambda)$ is $J^n$-$m$-linked.
\end{theorem}
\begin{proof}
Given $X\in (J^n)^+$ and $\langle p_a\in Add(\omega,\lambda): a\in X\rangle$, apply Theorem \ref{theorem: nDWeakDelta} to $\langle dom(p_a): a\in X\rangle$ and we get $J^n$-positive $Y\subset X$ and $W: [\kappa]^{\leq n-1}\to [\lambda]^{\aleph_0}$ such that for any $a\simeq b\in Y$, $dom(p_a)\cap dom(p_b)\subset W(a\cap b)$. We may assume there is some $z\in \omega$ and some function $l: z\to \omega$ such that for any $a\in Y$, $|dom(p_a)|=z$ and $p_a(dom(p_a)(j))=l(j)$ for all $j\in z$, where $dom(p_a)(j)$ is the $j$-th element of the natural enumeration of $dom(p_a)$. This is possible since $J$ is countably complete. From a natural numeration of $P(z)$, we can induce a natural enumeration of $P(a)$, as $\{a^i: i<2^z\}$ for any set $a$ of size $z$. For each $a\in Y$,  we record the following information: $\langle w^a_i: i<2^z\rangle$ where each $$w^a_i=\{(j, n): j<otp(W(a^i)), W(a^i)(j)\in dom(p_a), p_{a}(W(a^i)(j))=n\}.$$ Since $J^n$ is countably complete, we can find $J^n$-positive $Y'\subset Y$ and $\langle w_i: i<2^z\rangle$ such that for all $a\in Y'$, for all $i<2^z$, $w^a_i=w_i$. As a result, if $a\simeq b \in Y'$, then $p_a\cup p_b$ is a function. To see this, let $\eta\in dom(p_a)\cap dom(p_b)\subset W(a\cap b)$. Since $a\simeq b$, there is $i<2^z$ such that $c=a\cap b= a^i = b^i$. As a result, there is some $j$ such that $W(c)(j)=\eta$. Then we have $p_b(\eta)=n$ iff $(j, n)\in w_i$ iff $p_a(\eta)=n$. Suppose we have pairwise aligned $\{a^*_k: k<m\}\subset Y'$, namely for any $k_0,k_1<m$, $a^*_{k_0}\simeq a^*_{k_1}$. Then since for all $k,k'<m$, $p_{a^*_k}\cup p_{a^*_{k'}}$ is a function, we know that $\bigcup_{k<m} p_{a^*_k}$ is a function, as desired.
\end{proof}

Recall that a poset $P$ is \emph{$\sigma$-$m$-linked} for $m\in \omega$ if $P$ can be written as $\bigcup_{i\in \omega} P_i$ such that for each $i\in \omega$ and any $\langle r_j: j<m\rangle\subset P_i$, there is $r\in P$ with $r\leq_P r_j$ for any $j<m$.

\begin{theorem}\label{theorem: sigma-m-linked}
Fix a $\gamma$-saturated normal uniform ideal $J$ on $\kappa$ where $\gamma<\kappa$ and $m,n\in \omega$. Let $\langle P_\alpha, \dot{Q}_\beta: \alpha\leq \lambda, \beta<\lambda\rangle$ be a finite support iteration of $\sigma$-$m$-linked posets of length $\lambda$ for some ordinal $\lambda$. Then $P=P_\lambda$ is $J^n$-$m$-linked.
\end{theorem}

\begin{proof}
We can apply Theorem \ref{theorem: Cohen} granted that we have shown the following: 
there is a map $g: P_\lambda \to Add(\omega,\lambda)$ such that if $\langle p_i: i<m\rangle \subset P_\kappa$ satisfies that $\langle g(p_i): i<m\rangle$ admits a lower bound in $Add(\omega,\lambda)$, then $\langle p_i: i<m\rangle$ admits a lower bound in $P_\lambda$. It is enough to define such a $g$ on a dense subset of $P_\lambda$. For each $\beta\in \lambda$, let $\Vdash_{P_\beta} \dot{Q}_\beta=\bigcup_{n\in \omega} \dot{Q}_{\beta, n}$ witness that $\Vdash_{P_\beta}\dot{Q}_\beta$ is $\sigma$-$m$-linked. Then $D=\{p\in P_\lambda: \exists f=f_p: support(p) \to \omega, \forall \beta\in support(p), p\restriction \beta\Vdash_{P_\beta} p(\beta)\in \dot{Q}_{\beta, f(\beta)} \}$ is easily seen to be dense in $P_\lambda$. Then the map $p\in D\mapsto f_p\in Add(\omega, \lambda)$ is as desired.
\end{proof}

The following theorem can be thought of as a higher dimensional generalization of a result due to Fremlin \cite{fremlin}.

\begin{theorem}\label{theorem: Random}
For any $\gamma$-saturated normal uniform ideal $J$ on $\kappa$ where $\gamma<\kappa$, $m,n\in \omega$ and any measure algebra $\mathbb{M}$, $\mathbb{M}$ is $J^n$-$m$-linked. 
\end{theorem}
By Maharam's theorem \cite{maharam}, we may assume that $\mathbb{M}= (2^\lambda, \mu)$ where $\mu$ is the product $(\frac{1}{2},\frac{1}{2})$-probability measure. Fix such $\mathbb{M}$. For each $A\in \mathbb{M}$, there exist a countable $D\subset \lambda$ and $A'\subset 2^D$ such that for almost all $f\in 2^\lambda$, $f\in A$ iff $f\restriction D\in A'$. We call the $<^*$-least such $D$ \emph{the least support of $A$}, written as $D=supp(A)$, where $<^*$ is some well ordering of $H(2^{2^\lambda})$. Any $D$ satisfying above is called \emph{a support of $A$}. For each $A\in \mathbb{M}$ and a basic open set $U$, let $\mu_{U}(A)= \dfrac{\mu(A\cap U)}{\mu(U)}$ denote the measure $\mu$ relative to $U$. 

\begin{lemma}\label{lemma: enlarging}
Let $A\in \mathbb{M}$ and $U$ be a basic open set. Suppose $V$ is another basic open set. Let $S_A, S_U, S_V$ be some supports for $A, U ,V$ respectively. If $S_A\cap S_V =S_U\cap S_V=\emptyset$. Then $\mu_{U}(A)=\mu_{U \cap V} (A)$.
\end{lemma}

\begin{proof}
Since $B=S_A\cup S_U$ is a support for $A\cap U$, by the hypothesis, we know that $B\cap S_V=\emptyset$. Therefore, $\mu(A\cap U \cap V)= \mu(A\cap U)\cdot \mu(V)$. Similarly, $\mu(U\cap V)=\mu(U)\cdot \mu(V)$.
Hence $\mu_{U\cap V}(A)=\dfrac{\mu(A\cap U\cap V)}{\mu(U\cap V)}=\dfrac{\mu(A\cap U)\cdot \mu(V)}{\mu(U)\cdot \mu(V)}=\mu_U(A)$.
\end{proof}

\begin{lemma}\label{lemma: intersection}
For each $m\in \omega-\{0\}$ and $\epsilon\in (0,1)$, there exists some $f(m,\epsilon)\in (0,1)$ such that for any $\{A_i: i<m\}\subset \mathbb{M}$ such that $\mu(A_i)>f(m,\epsilon)$, then $\mu(\bigcap_{i<m} A_i)>\epsilon$.
\end{lemma}

\begin{proof}
Let us verify that $f(m,\epsilon)=\frac{\epsilon+m-1}{m}$ works. Given $\langle A_i: i<m\rangle$ with $\mu(A_i)>f(m,\epsilon)$, we know that $\mu(\bigcup_{i<m} A_i^c)\leq \Sigma_{i<m} \mu(A_i^c) < \Sigma_{i<m} 1-\frac{\epsilon+m-1}{m}=1-\epsilon$. Therefore, $\mu(\bigcap_{i<m} A_i)=\mu((\bigcup_{i<m} A_i^c)^c)>\epsilon$.
\end{proof}

\begin{proof}[Proof of Theorem \ref{theorem: Random}]
Let $X\in (J^n)^+$ and $\langle F_{a}\in \mathbb{M}: a\in X\rangle$ be given. For each $a\in X$, let $A_a=supp(F_a)$. We will slightly abuse the notations in the following in that for $F\subset \lambda$, we naturally identify a subset $B\subset 2^F$ as a subset in $2^\lambda$, which is $B'=\{f\in 2^\lambda: f\restriction F\in B\}$. Hence when we write $\mu(B)$, we really mean $\mu(B')$.

Apply Theorem \ref{theorem: nDWeakDelta} to $\langle A_a: a\in X\rangle$ to get $J^n$-positive $Y\subset X$ and $W: [\kappa]^{\leq n-1}\to [\lambda]^{\aleph_0}$ such that for any $a\simeq b\in Y$, $A_a\cap A_b\subset W(a\cap b)$.
 
By shrinking $Y$ to a $J^n$-positive subset if necessary, we may assume that there is $\delta<\omega_1$ such that each $A\in range(W\restriction Y)$ satisfies $otp(A)\leq \delta$. The reason is that for any $a\in Y$, we can associate it with $\delta_a=otp(W(a))<\omega_1$. As $J^n$ is $\aleph_2$-complete, there is a positive subset of $Y$ on which the function $a\mapsto \delta_a$ is constant.

 For each such $A\in range(W\restriction Y)$, we let $\langle A(k): k<otp(A)\rangle$ be its increasing enumeration.
For each $a\in Y$, apply the Lebesgue density theorem to find $s_a$ which is a finite partial function from $A_a$ to $2$ such that $\dfrac{\mu(F_a\cap [s_a])}{\mu([s_a])}>f(m,0.1)$, where $f$ is the function given by the conclusion of Lemma \ref{lemma: intersection}. From a canonical enumeration of $P(n)$, we can induce a canonical enumeration $\langle a^i: i<2^n\rangle$ of $P(a)$ for any $a$ of size $n$.
For each $a\in Y$ and for each $i<2^n$, define $w_i^a$ records the index of $dom(s_a)\cap W(a^i)$. Namely, $w_i^a=\{(j,n): j<\delta, W(a^i)(j)\in dom(s_a), s_a(W(a^i)(j))=n\}$. Find $J^n$-positive $Y'\subset Y$, $\langle w_i: i<2^n\rangle$ and $z\in \omega$ such that for each $a \simeq a'\in Y'$ 
\begin{itemize}
\item $|s_a|=z$,
\item for each $i<2^n$, $w^a_i=w_i$.
\end{itemize} 
We can do this by the $\sigma$-completeness of $J^n$. Given $\langle a_l: l<m\rangle\subset Y'$ with $a_l\simeq a_{l'}$ for any $l,l'<m$, we demonstrate that $\mu(\bigcap_{l<m} F_{a_l})>0$. 

Let $s^*=\bigcup_{l<m}s_{a_l}$ and $V=\bigcap_{l<m}[s_{a_l}]=[s^*]$. Note that  $s^*$ is a function, reasoning as in Theorem \ref{theorem: Cohen}.

We next show that $\mu_V(F_{a_l})=\mu_{[s_{a_l}]}(F_{a_l})$ for any $l<m$. To see this, fix $l<m$.  

	\begin{claim}
	$(dom(s^*)-dom(s_{a_l}))\cap A_{a_l} =\emptyset$.
	\end{claim}
	\begin{proof}
	Suppose for the sake of contradiction that there is some $\eta\in A_{a_l} \cap  (dom(s_{a_{l'}})-dom(s_{a_l}))$. As a result, $\eta\in A_{a_l}\cap A_{a_{l'}}\subset W(a_{l}\cap a_{l'})$. Since $a_l\simeq a_{l'}$, there is some $i<2^n$ such that $a^*=a_{l}\cap a_{l'}=(a_l)^i=(a_{l'})^i$. Therefore, there exists $j<\delta$ such that $W(a^*)(j)=\eta$. In particular, $(j, s^*(\eta))\in w_i$. Since $w_i=w_i^{a_l}$, we know that $W((a_l)^i)(j)=W(a^*)(j)=\eta\in dom(s_{a_l})$, which is a contradiction.
	\end{proof}
	
By the previous claim, we can apply Lemma \ref{lemma: enlarging} to $F_{a_l}$, $[s_{a_l}]$, $[s^*\restriction (dom(s^*)-dom(s_{a_l}))]$ to conclude that $\mu_V(F_{a_l})=\mu_{[s_{a_l}]}(F_{a_l})$. Finally, apply Lemma \ref{lemma: intersection} relative to $V$, we get $\mu_V(\bigcap_{l<m}F_{a_l})>0.1>0$, as desired.\end{proof} 

Theorem \ref{theorem: CohenRandom} follows from Theorem \ref{theorem: sigma-m-linked} and \ref{theorem: Random}.

We finish this section by showing if the quotient algebra by an ideal on $\kappa$ satisfies the higher dimensional chain conditions, then $\kappa$ satisfies a certain Ramsey-type properties that have implications in the vanishing of higher derived limits of a certain inverse system. We refer the readers to \cite{bannister2023descriptive} and \cite{MR4275058} for more information concerning the historical background and the homological applications. The following definitions are from \cite{bannister2023descriptive}.

\begin{definition}
A function $F: [\kappa]^{\leq n}\to \kappa$ is cofinal if 
	\begin{enumerate}
	\item for any $\alpha\in \kappa$, $\alpha < F(\alpha)$,
	\item for any $x \subsetneq y$, $F(x) < F(y)$.
	\end{enumerate}
\end{definition}

\begin{definition}
Let $[\kappa]^{[n]}$ denote the set $$\{ \langle \sigma(i): i<n\rangle | \sigma(i)\in [\kappa]^{i+1}, \sigma(i)\subset \sigma(i+1) \text{ for all }i<n-1\}.$$
\end{definition}

Given $a\in [\kappa]^n$, let $\bar{\sigma}_a $ be the collection of $\bar{\sigma}\in [\kappa]^{[n]}$ such that $a =\bigcup_{i<n} \sigma(i)$. We call $\bar{\sigma}\in \bar{\sigma}_a$ an \emph{enumeration} of $a$.

\begin{definition}
Given $F: [\kappa]^{\leq n} \to \kappa$, define $F^*: [\kappa]^{[n]}\to [\kappa]^n$ as follows: $F^*(\sigma)=\{F(\sigma(i)): i<n\}$
\end{definition}

\begin{definition}\label{definition: PH}
$\mathrm{PH}_n(\kappa)$ abbreviates: for any $c: [\kappa]^{n+1}\to \omega$, there exists a cofinal $F: [\kappa]^{\leq n+1}\to \kappa$, such that $c\circ F^* $ is constant.
\end{definition}

\begin{theorem}\label{theorem: PH}
Let $I$ be a uniform normal ideal on $\kappa$. Suppose that $P(\kappa)/I$ is $I^n$-$m$-linked for all $n,m\in \omega$.
	Then $\mathrm{PH}_n(\kappa)$ holds for all $n\in \omega$.
\end{theorem}

\begin{proof}
Fix a coloring $c: [\kappa]^{n+1}\to \omega$.
For each $a\in [\kappa]^n$, by the $\sigma$-completeness of $I$, find $i_a\in \omega$, such that $X_a =\{\gamma: c(a\cup \{\gamma\})=i_a\}\in I^+$. Let $l=n!$.

By the hypothesis on $P(\kappa)/I$, we can find $X\in (I^n)^+$ such that
\begin{itemize}
\item there is $i\in \omega$, for all $a\in X$, $i_a=i$, 
\item for any $\{a_i: i<l\} \subset X$ such that $a_i\simeq a_j$ for any $i,j<l$, $\bigcap_{i<l} X_{a_i} \in I^+$, and 
\item for any $k<n$ and any $\langle c_i: i<l\rangle \subset X\restriction k$ such that $c_i\simeq c_j$ for any $i,j<l$, then $\bigcap_{i<l}(X\restriction k+1)_{c_i}\in I^+$.
\end{itemize}

Let $<_l$ be the following ordering on $[\kappa]^{\leq n+1}$ defined recursively: $a<_l b$ iff $a\neq b$ and either $a=\emptyset$ or $\max a < \max b$ or $\max a =\max b$ and $a-\{\max a\}<_l b-\{\max b\}$.

We will define $F$ on $[\kappa]^{\leq n+1}$ recursively such that in addition to being cofinal it satisfies the following: 

	\begin{enumerate}
	\item for $a<_l b$, $F(a)<F(b)$,
	\item for any $k\leq n+1$, $a \in [\kappa]^{k}$ and $\bar{\sigma}\in \bar{\sigma}_a$, $F(a)\in (X\restriction k)_{\langle F(\sigma(i)): i<k-1\rangle}$
	\end{enumerate}

Suppose the construction is successful. Let us check that $c\circ F^*\equiv \{i\}$. Given some $\bar{\sigma}\in [\kappa]^{[n+1]}$, $c(F^*(\bar{\sigma}))=c(\langle F(\sigma(i)): i<n+1\rangle)$. By property (2), we know that $F(\sigma(n))\in (X)_{F(\sigma(0)), F(\sigma(1)),\cdots, F(\sigma(n-1))}$, which means $c(F(\sigma(0)), F(\sigma(1)),\cdots, F(\sigma(n-1)), F(\sigma(n)))=i$ as desired.

Finally, let us check that the construction can be carried out. Suppose we are at the stage of defining $F(a)$ for $a\in [\kappa]^k$ where $F(c)$ for $c<_l a$ has been defined and the two properties are satisfied.
It suffices to show that $\bigcap_{\bar{\sigma}\in \bar{\sigma}_a} (X\restriction k)_{\langle F(\sigma(i)): i<k-1\rangle}\in I^+$. Since we can then define $F(a)$ to be a large enough element in this intersection.
By the property of $X$, we only need to show for any $\bar{\sigma}_0, \bar{\sigma}_1\in \bar{\sigma}_a$, $\langle F(\sigma_0(i)): i<k-1\rangle \simeq \langle F(\sigma_1(i)): i<k-1\rangle$. To see this, note that by property (1), $F \restriction \{c: c<_l a\}$ is strictly increasing. Hence if there are $h,l$ such that $F(\sigma_0(h))=F(\sigma_1(l))$, it must be that $h=l$ and $\sigma_0(h)=\sigma_1(l)$. As a result, $\langle F(\sigma_0(i)): i<k-1\rangle \simeq \langle F(\sigma_1(i)): i<k-1\rangle$ as desired.
\end{proof}

\begin{corollary}\label{cor: applicationRandom}
Let $J$ be a uniform normal ideal on $\kappa$ such that $P(\kappa)/J \simeq \mathbb{M}$ where $\mathbb{M}$ is a non-trivial measure algebra,
then $\mathrm{PH}_n(\kappa)$ holds for all $n\in \omega$. In particular, if $\kappa$ is real-valued measurable, then $\mathrm{PH}_n(\kappa)$ holds for all $n\in \omega$.
\end{corollary}

\begin{proof}
This follows from Theorem \ref{theorem: Random} and Theorem \ref{theorem: PH}.
\end{proof}

\begin{corollary}\label{cor: applicationSigmaCentered}
Let $\kappa$ be a measurable cardinal and $P$ be a finite-support iteration of $\sigma$-centered posets of length $\kappa$. Then in $V^{P}$, $\mathrm{PH}_n(\kappa)$ holds for all $n\in \omega$.
\end{corollary}

\begin{proof}
Fix $\kappa$-complete normal measure $U$ on $\kappa$. Let $j: V\to M\simeq Ult(V,U)$ be the ultrapower embedding derived from $U$.
Let $P$ be a finite-support iteration of $\sigma$-centered posets of length $\kappa$. Let $G\subset P$ be generic over $V$. We show that in $V[G]$, letting $J$ be the ideal on $\kappa$ defined by $X\in J$ iff $\Vdash_{j(P)/G} \kappa\not\in j(\dot{X})$. It is routine to check that $J$ is $\kappa$-complete, normal and $\aleph_1$-saturated. By the elementarity of $j$, $M[G]\models j(P)/G$ is a finite-support iteration of $\sigma$-centered posets. As in the proof of Theorem \ref{theorem: sigma-m-linked}, in $M[G]$ there exists a map $g: j(P)/G\to Add(\omega, j(\kappa))$ such that any finite $\{p_i:i<k\}\subset j(P)/G$ admits a lower bound in $j(P)/G$ whenever $\{g(p_i): i<k\}$ admits a lower bound in $Add(\omega, j(\kappa))$. It is easy to see that such property of $g$ also holds in $V[G]$. Applying Theorem \ref{theorem: Cohen} to $Add(\omega, j(\kappa))$ and $J$ in $V[G]$, we can conclude that $j(P)/G$ is $J^n$-$m$-linked for all $m,n\in \omega$. Finally, we show that $P(\kappa)/J\simeq \mathbb{B}(j(P)/G)$. Consider the map $f: P(\kappa)/J\to \mathbb{B}(j(P)/G)$ such that $f(X)=\llbracket \kappa\in j(\dot{X}) \rrbracket$. It is clear that $f$ is an order preserving embedding. We check that the image of $f$ is dense in $\mathbb{B}(j(P)/G)$, which is clearly sufficient. Let $p\in j(P)/G$. Fix some $l\in V$ such that $p=j(l)(\kappa)$. Consider $X=\{\alpha<\kappa: l(\alpha)\in G\}$. Note that $X\in J^+$ since $p\Vdash_{j(P)/G} \kappa\in j(\dot{X})$. We claim that $f(X)\leq_{\mathbb{B}(j(P)/G)} p$. Let $H\subset \mathbb{B}(j(P)/G)$ be generic over $V[G]$ containing $f(X)$. Then in $V[G*H]$, letting $j^+: V[G]\to M[G*H]$ be the canonical lift of $j$, we get that $\kappa\in j^+(X)= (j(\dot{X}))^{G*H}$. By the definition of $X$, we know that $p=j(l)(\kappa)\in j^+(G)=G*H$. As a result, $f(X)\leq_{\mathbb{B}(j(P)/G)} p$.
To summarize, we have shown that in $V[G]$, $J$ is a uniform normal ideal on $\kappa$ such that $P(\kappa)/J\simeq \mathbb{B}(j(P)/G)$ is $J^n$-$m$-linked for all $m,n\in \omega$. Apply Theorem \ref{theorem: PH} to finish.\end{proof}
Previously in \cite{bannister2023descriptive} building on \cite{MR4275058}, it was shown that $\mathrm{PH}_n(\kappa)$ holds for all $n\in \omega$ in the forcing extension by any $\kappa$-length finite support iteration of Hechler forcing over a ground model where $\kappa$ is a weakly compact cardinal.

To finish this section, we remark on the situation on longer iterations. 

\begin{definition}
A poset $P$ is \emph{$\sigma$-centered-plus} if 
	\begin{enumerate}
	\item $P$ is $\sigma$-centered witnessed by $\bigcup_{n\in\omega} P_n$;
	\item for any $p\in P$, $n\in \omega$ and finite $\{q_i: i<k\}\subset P_n$, if $p || q_i$ for each $i<k$, then there is a lower bound for $\{p, q_i: i<k\}$;
	\item for any $p || q\in P$, there exists a greatest lower bound $p\wedge q\in P$. Namely, there is a condition $r\leq p,q$ such that for any $r'\leq p, q$, we have that $r'\leq r$.
	\end{enumerate}
\end{definition}

Examples of $\sigma$-centered-plus posets include any Cohen forcing for adding $\leq 2^\omega$ subsets of $\omega$, Hechler forcing, ultrafilter Mathias forcing and so on.

\begin{corollary}\label{cor: longer}
Let $\kappa$ be a measurable cardinal and $P$ be a finite-support iteration of $\sigma$-centered-plus posets of length $\lambda$ for some ordinal $\lambda$. Then in $V^P$, $\mathrm{PH}_n(\kappa)$ holds for all $n\in \omega$.
\end{corollary}
\begin{proof}
We may assume $\lambda>\kappa$.
We follow the notations and argument as in Corollary \ref{cor: applicationSigmaCentered}. The main difference, which is also the reason for considering the plus-strengthening of $\sigma$-centered-ness, is that we need to show that in $V[G]$, $j(P)/j''G$ is $J^n$-$m$-linked for all $n, m\in \omega$ but $j'' P $ is not necessarily a member of $M$. Hence, the previous argument using the elementarity of $j$ does not work. To overcome this, we proceed as follows. Let $g\in M$ be the function $g: j(P)\to Add(\omega, j(\lambda))$ as in the proof of Theorem \ref{theorem: sigma-m-linked}. Namely, $\{q_i\in j(P): i<k\}$ admits a lower bound whenever $\{g(q_i): i<k\}$ admits a lower bound in $Add(\omega,j(\lambda))$. We claim that in $V[G]$, the property above for $g$ still holds when $j(P)$ is replaced with $j(P)/j''G$ everywhere. 
Given $\bar{q}=\{q_i: i<k\}\subset j(P)/j''G$, with $r_i=g(q_i)$ and $r^*=\bigcup_{i < k} r_i\in Add(\omega,j(\lambda))$, we need to find a lower bound in $j(P)/j''G$. 

Let us work in $V$ and suppose otherwise for the sake of contradiction. Let $p\in P$ forcing all the statements above about $\bar{q}$ and that $\bar{q}$ does not admit a lower bound in $j(P)/j''\dot{G}$. We will recursively define some $r\in j(P)$ extending $\{j(p), q_i: i<k\}$. This is enough since if $H\subset j(P)$ is $V$-generic containing $r$, then $G'=j^{-1}(H)$ is $V$-generic containing $p$. But then $r$ is a lower bound for $\bar{q}$ in $j(P)/j''G'$, contradicting with our assumption. Let us construct such $r$ recursively. We will make sure that for each $\alpha<j(\lambda)$, 
\begin{enumerate}
\item $r\restriction \alpha\leq j(p)\restriction \alpha, q_i\restriction \alpha$ for all $i<k$, and
\item $r\restriction \alpha \Vdash_{j(P)\restriction \alpha} j(p)\downharpoonright \alpha$ is compatible in $j(P)\downharpoonright \alpha$ with $q_i\downharpoonright \alpha$ for any $i<k$.
\end{enumerate}
The construction when $\alpha\not\in support(j(p))$ is clear. Let us demonstrate the construction when $\alpha\in support(j(p))$, in which case there is $\beta$ such that $j(\beta)=\alpha$. First note that $r\restriction \alpha\Vdash j(p)(\alpha)$ is compatible with $\{q_i(\alpha): i<k\}\subset \dot{Q}_{\alpha,r^*(\alpha)}$. By the $\sigma$-centered-plus assumption, we can find a $j(P)\restriction \alpha$-name $\dot{\tau}\in \dot{Q}_\alpha$ such that $r\restriction \alpha \Vdash\dot{\tau}\leq \{j(p)(\alpha), q_i(\alpha): i<k\}$. Let us define $r\restriction \alpha+1$ to be $r\restriction \alpha ^\frown \dot{\tau}$. Let us verify the second requirement. By an inductive argument, one can show that $r\restriction \alpha \Vdash$ ``the greatest lower bound $\dot{\sigma}_i=(j(p)\downharpoonright \alpha) \wedge (q_i\downharpoonright \alpha)$ exists in $j(P)\downharpoonright\alpha$ for any $i<k$".
As a result, $r\restriction \alpha \Vdash \dot{\tau}\leq \dot{\sigma}_i(\alpha)$ for all $i<k$. Therefore, $r\restriction \alpha+1 \Vdash j(p)\downharpoonright \alpha+1$ is compatible in $j(P)\downharpoonright \alpha+1$ with $q_i\downharpoonright \alpha+1$ for any $i<k$.\end{proof}

\section{A rectangular strengthening of the $\sigma$-finite-chain-condition}\label{section: wedge}

In this section we are motivated by the following natural question: will satisfying a certain 1-dimensional strong chain condition have higher dimensional chain condition consequences?

Recall the following definition.

\begin{definition}[Horn and Tarski, \cite{horntarski}]
A poset $P$ satisfies \emph{$\sigma$-finite-c.c} if there exists a decomposition $P=\bigcup_{n\in \omega} P_n$ such that for each $n\in \omega$, $P_n$ does not contain an infinite antichain in $P$.
\end{definition}

\begin{definition}\label{definition: compatibility}
Let $P$ be a poset and $p,q\in P$. Let $p || q$ denote $p$ and $q$ are compatible in $P$, and let $p\perp q$ denote that $\neg (p || q)$.
\end{definition}

This question is indeed natural in light of the following propositions.

\begin{proposition}\label{proposition: prop1d}
Any $\sigma$-finite-c.c poset $Q$ is $J$-Knaster, where $J$ is a uniform normal ideal on a regular uncountable cardinal $\lambda$.
\end{proposition}
\begin{proof}
Let $\bigcup_{n\in \omega} Q_n=Q$ witness the $\sigma$-finite-c.c-ness of $Q$.
Given $X\in J^+$ and $\langle q_\alpha\in Q: \alpha\in X\rangle$, we need to find $J$-positive $Y\subset X$ such that for any $\alpha<\beta\in Y$, $q_\alpha || q_\beta$. We may assume without loss of generality that there is some $m\in \omega$ such that $q_\alpha\in Q_m$ by the $\sigma$-completeness of $J$.
For any $Z\in J^+$ and $\alpha<\lambda$, we say $\alpha$ \emph{is good for $Z$} if there exists $C\in J^*$ such that for all $\beta \in C\cap Z$, $q_\alpha || q_\beta$. 
	\begin{claim}
	There exists a $J$-positive $Y'\subset X$ such that any $\alpha\in Y'$ is good for $Y'$.
	\end{claim}
	\begin{proof}[Proof of the claim]
	Suppose otherwise for the sake of contradiction. We construct a sequence $\langle \alpha_i, X_i: i\in \omega\rangle$ such that 	
	\begin{enumerate}
	\item $X_{-1}=X$,
	\item $X_i\in J^+$ for all $i\in \omega$ and $X_{i+1}\subset X_i$,
	\item $\alpha_i\in X_{i-1}$ is not good for $X_{i-1}$ as witnessed by $X_i$. Namely, for any $\beta\in X_i$, $q_{\alpha_i}\perp q_\beta$.
	\end{enumerate}
	The construction is possible by the hypothesis. As a result, $\{q_{\alpha_i}: i\in \omega\}\subset Q_m$ is an infinite antichain in $Q$, which is a contradiction.
	\end{proof}
Fix $Y'\subset X$ as in the claim. For each $\alpha\in Y'$, find $C_\alpha\in J^*$ such that for all $\beta\in C_\alpha\cap Y'$, $q_\alpha || q_\beta$. Let $Y=Y'\cap \Delta_{\alpha\in Y'} C_\alpha\in J^+$ by the normality of $J$. It is easy to see that $Y$ is as desired.
\end{proof}

\begin{proposition}
Let $I$ be a normal uniform ideal on $\kappa$ and $J_\lambda=[\lambda]^{<\lambda}$ be the bounded ideal on $\lambda\leq \kappa$, where $\lambda, \kappa$ are regular cardinals $>\omega$.
If $P(\kappa)/I$ satisfies $\sigma$-finite-c.c, then $P([\kappa]^n)/I^n$ is $J_\lambda$-Knaster (or more commonly known as $\lambda$-Knaster) for all $n\in \omega$. 
\end{proposition}

\begin{proof}
Let us prove this by induction on $n\geq 1$ for all ideals $I$. Proposition \ref{proposition: prop1d} handles the case when $n=1$. More precisely, given $\langle X_\alpha\in P(\kappa)/I: \alpha\in A\rangle$ where $A\in [\lambda]^\lambda$, by re-indexing, we may assume $A=\lambda$. Then apply Proposition \ref{proposition: prop1d} to $\langle X_\alpha\in P(\kappa)/I: \alpha\in \lambda\rangle$ and the non-stationary ideal on $\lambda$. 

In general, Fix $n=k+1$ where $k\geq 1$. 
	Given $X\in [\lambda]^\lambda$ and $\langle X_\alpha\in (I^n)^+: \alpha\in X\rangle$, let $Y'_\alpha=\{\beta\in \kappa: (X_\alpha)_\beta\in (I^k)^+\}\in I^+$. Let $G\subset P(\kappa)/I$ be generic over $V$ such that $X'=\{\alpha\in X: Y'_\alpha\in G\}\in [\lambda]^\lambda$. This is possible by the fact that $P(\kappa)/I$ is $\lambda$-Knaster. Let $j: V\to M\simeq \mathrm{Ult}(V, G)$ be the generic elementary embedding. For each $\alpha\in X'$, let $Y^*_\alpha= (j(X_\alpha))_\kappa$. In particular, $Y^*_\alpha\in (P([j(\kappa)]^\kappa)/j(I^k))^M$. By the elementarity of $j$, $M\models P(j(\kappa))/j(I)$ satisfies $\sigma$-finite-c.c.  
As $P(\kappa)/I$ is c.c.c, $\lambda$ remains uncountable regular. Apply the induction hypothesis in $M$, we get $B\in [X']^\lambda$ such that for any $\alpha<\beta\in B$, $Y^*_\alpha\cap Y^*_\beta \in j((I^k)^+)$. Back to $V$. For each $\alpha\in X$, find $I$-positive $Y_\alpha\subset Y'_\alpha$ forcing that $\alpha\in \dot{B}$ if such set exists. Otherwise, $Y_\alpha$ is undefined. Note that $Z=\{\alpha\in X: Y_\alpha \text{ is defined}\}\in [X]^\lambda$. As $P(\kappa)/I$ is $\lambda$-Knaster, we can find $J_\lambda$-positive $Y\in [Z]^\lambda$ such that for any $\alpha<\beta\in Y$, $Y_\alpha\cap Y_\beta\in I^+$. We claim that $Y$ is as desired. To see this, for any $\alpha<\beta\in Y$, $Y_\alpha\cap Y_\beta$ forces that $\alpha,\beta\in\dot{B}$. As a result, $Y_\alpha\cap Y_\beta$ forces that $j(X_\alpha)_\kappa\cap j(X_\beta)_\kappa\in j((I^k)^+)$. Elementarity implies that $\{\eta\in Y_\alpha\cap Y_\beta: (X_\alpha)_\eta\cap (X_\beta)_\eta\in (I^k)^+\}\in J^*\restriction Y_\alpha\cap Y_\beta$. In particular, $X_\alpha\cap X_\beta \in (I^n)^+$. \end{proof}

It is a natural question whether we can prove the following stronger statement: ``if $P(\kappa)/I$ satisfies $\sigma$-finite-c.c, then $I$ is $J_\kappa^n$-Knaster for all $n\in \omega$"? While we do not know the answer to this question (see Section \ref{section: openquestions}), we do manage to isolate a strengthening of $\sigma$-finite-c.c-ness that admits such implications.

\begin{definition}\label{definition: sigmawedgefinitecc}
A poset $P$ is \emph{$\sigma$-w-finite-c.c} if there exists $P=\bigcup_{k\in \omega} P_k$ such that for any $k\in \omega$, 
\begin{itemize}
\item $P_k\subset P_{k+1}$, and
\item for any $\langle (p_i,q_i): i\in \omega\rangle\subset P_k\times P_k$ satisfying that $p_i$ is compatible with $q_i$ for each $i\in \omega$, there exist $l<j$ such that $p_l$ is compatible with $q_j$. 
\end{itemize}
\end{definition}

To explain the ``$w$" in our naming choice, notice that the conclusion insists that the configuration of ``$v$" appears in any given legitimate sequence. An easy application of Ramsey theorem gives that the configuration actually appears many times. ``$w$" is chosen to emphasize the repetition of such patterns. The following is an illustration. The red line between two conditions indicates their compatibility.

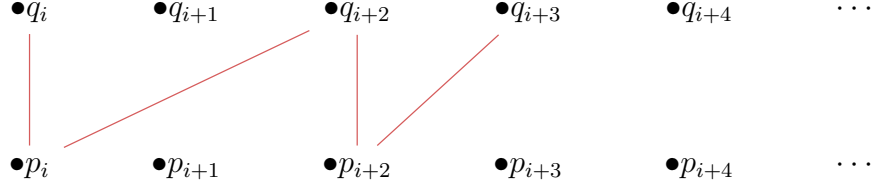
\[\begin{tikzcd}
	{\bullet{q_i}} & {\bullet{q_{i+1}}} & {\bullet {q_{i+2}}} & {\bullet{q_{i+3}}} & {\bullet{q_{i+4}}}  & \cdots \\
	\\
	{\bullet{p_{i}}} & {\bullet{p_{i+1}}} & {\bullet{p_{i+2}}} & {\bullet{p_{i+3}}} & {\bullet{p_{i+4}}} &  \cdots
	\arrow[draw={rgb,255:red,214;green,92;blue,92}, no head, from=1-1, to=3-1]
	\arrow[draw={rgb,255:red,214;green,92;blue,92}, no head, from=1-3, to=3-3]
	\arrow[draw={rgb,255:red,214;green,92;blue,92}, no head, from=3-1, to=1-3]
	\arrow[draw={rgb,255:red,214;green,92;blue,92}, no head, from=3-3, to=1-4]
\end{tikzcd}\]

\begin{observation}
	 If $P$ satisfies $\sigma$-$w$-finite-c.c, then $P$ satisfies $\sigma$-finite-c.c: apply the property to $p_i=q_i$ as in the above.
\end{observation}

We give some non-trivial examples and non-examples of $\sigma$-$w$-finite-c.c posets. 

\begin{proposition}\label{prop: CohenExample}
For any $\lambda$, the Cohen forcing for adding $\lambda$ many subsets of $\omega$, $Add(\omega,\lambda)$, is $\sigma$-$w$-finite-c.c.
\end{proposition}

\begin{proof}
Fix $\lambda$. For each $r\in 2^{<\omega}$, define $C_r=\{p\in Add(\omega, \lambda): p\simeq r\}$. Here $p\simeq r$ means 
	\begin{itemize}
	\item $|support(p)|=dom(r)$,
	\item for any $i<dom(r)$, $p(support(p)(i))=r(i)$ where $support(p)(i)$ is the $i$-th element of $support(p)$ in the order of ordinals.
	\end{itemize}

Let $\langle r_n: n\in \omega\rangle$ be an enumeration of $2^{<\omega}$. Let $D_{m}=\bigcup_{i\leq m} C_{r_i}$. We show that $\bigcup_{m\in \omega} D_m = Add(\omega, \lambda)$ witnesses that $Add(\omega, \lambda)$ is $\sigma$-$w$-finite-c.c. 

To see this, let $\langle (p_i,q_i): i\in \omega\rangle \subset D_m\times D_m$ satisfying that $p_i || q_i$ for any $i\in \omega$ be given. Since the sizes of the domain of the relevant conditions are bounded above by $\max \{dom(r_j): j\leq m\}$, we can find an infinite $A\in [\omega]^{\aleph_0}$ such that 
	\begin{itemize}
	\item $\{support(p_i): i\in A\}$ forms a $\Delta$-system with root $r_p$ and there is $a: r_p\to 2$ such that each $p_i\restriction r_p = a$ for any $i\in A$, 
	\item  $\{support(q_i): i\in A\}$ forms a $\Delta$-system with root $r_q$ and there is $b: r_q\to 2$ such that each $q_i\restriction r_q = b$ for any $i\in A$. 
	\end{itemize}
Fix $i\in A$. Let us find some large enough $j\in A$ such that $support(q_j)\cap support(p_i)=r_q\cap support(p_i)$. The reason why this is possible is that for any $\alpha\in support(p_i)-r_q$, there is at most one $l\in A$ such that $\alpha\in support(q_l)$. Fix such $i<j\in A$. For any $\alpha\in support(p_i)\cap support(q_j)\subset r_q$, $q_j(\alpha)=b(\alpha)=q_i(\alpha)=p_i(\alpha)$, where the last equality is due to the fact that $p_i$ is compatible with $q_i$. We have then verified that $p_i || q_j$. This completes our proof.
\end{proof}

\begin{remark}
Proposition \ref{prop: CohenExample} and the forthcoming Theorem \ref{theorem: main} can be used to give another proof that if $P$ is a finite support iteration of $\sigma$-linked forcings of any length and $J$ is a uniform normal $\nu$-saturated ideal on $\kappa$ with $\nu<\kappa$, then $P$ is $J^n$-Knaster. To see this, fix $g: P\to Add(\omega,\lambda)$ as in Theorem \ref{theorem: sigma-m-linked} satisfying that whenever $g(p)$ and $g(q)$ are compatible in $Add(\omega,\lambda)$, $p$ and $q$ are compatible in $P$. That $P$ is $J^n$-Knaster follows from the fact that $Add(\omega,\lambda)$ is $J^n$-Knaster and such property is preserved by pulling back with $g$.
\end{remark}

The next proposition shows that the $\sigma$-$w$-finite chain condition is a proper strengthening of the $\sigma$-finite chain condition.
\begin{proposition}
Non-trivial measure algebras are not $\sigma$-$w$-finite-c.c.
\end{proposition}

\begin{proof}
Let $(\mathbb{M},\mu)$ be some non-trivial measure algebra. For concreteness, we can take $\mathbb{M}$ to be $2^\omega$ with the usual product topology and $\mu$ is the $(\frac{1}{2}, \frac{1}{2})$-probability measure.
Suppose for the sake of contradiction that $\mathbb{M}=\bigcup_{n\in \omega} M_n$ witnesses that $\mathbb{M}$ is $\sigma$-$w$-finite-c.c. We say $A\in \mathbb{M}$ is \emph{positive} if $\mu(A)>0$.

	\begin{claim}\label{claim: iterate}
For any positive $D\in \mathbb{M}$, there exist disjoint positive sets $A, C\subset D$ and $n\in \omega$ such that for any positive $C'\subset C$, there is positive $C''\subset C'$ such that $A\cup C''\in P_n$.
	\end{claim}
	\begin{proof}[Proof of the Claim]
	Otherwise, suppose some positive $D$ is a counter example. We will recursively define disjoint positive subsets $\{D^i_0: i\in \omega\}$ of $D$ as follows. At stage $0$, partition $D$ into disjoint positive $D^0_0, E^0_1$. By the hypothesis applied to $D^0_0, E^0_1$ and $0$, there is a positive $D^0_1\subset E^0_1$ such that for any positive $B'\subset D^0_1$, we have that $B'\cup D^0_0 \not \in M_0$. 
At stage $n+1$, suppose we have defined positive sets $\langle D^i_0, D^i_1: i\leq n\rangle$ satisfying for any $i\leq n$,
	\begin{enumerate}
	\item $D^i_0\cap D^i_1 =\emptyset$,
	\item $D^{j}_0, D^j_1\subset D^i_1$ for any $i<j\leq n$,
	\item for any positive $B'\subset D^i_1$, $B'\cup \bigcup_{l\leq n} D^{l}_0 \not \in M_i$.
	\end{enumerate}

Partition $D^n_1$ into positive $D^{n+1}_0, E^{n+1}_1$. By the hypothesis applied to $\bigcup_{l\leq n+1} D^{l}_0$, $E^{n+1}_1$ and $n+1$, there is some positive $D^{n+1}_1\subset E^{n+1}_1$ such that for any positive $B'\subset D^{n+1}_1$, it is the case that $B'\cup \bigcup_{l\leq n+1} D^{l}_0 \not \in M_{n+1}$. 

This process gives us a sequence of disjoint positive subsets of $D$, $\langle D^i_0: i\in \omega\rangle$ satisfying that for any $m\in \omega$, $\bigcup_{l\in \omega} D^l_0=(\bigcup_{l>m} D^l_0)\cup \bigcup_{l\leq m} D^l_0\not\in M_m$, which is a contradiction.
	\end{proof}
Apply Claim \ref{claim: iterate} twice to get disjoint positive sets $A, B, C\in \mathbb{M}$ and $m\in \omega$ satisfying: 
	\begin{enumerate}
	\item for any positive $C'\subset C$, there is a positive $C''\subset C$ such that $A\cup C''\in M_m$, 
	\item for any positive $C'\subset C$, there is a positive $C''\subset C$ such that $B\cup C''\in M_m$.
	\end{enumerate}
	
Partition $C$ into countably infinite pairwise disjoint positive sets $\langle C_n: n\in \omega\rangle$. For each $C_n$, by the hypothesis, we can find positive sets $C_n^0\subset C_n^1 \subset C_n$ such that $A\cup C^0_n \in M_m $ and $ B\cup C^1_n\in M_m$. 

Finally, consider $\langle (A\cup C^0_n, B\cup C^1_n): n\in \omega\rangle \subset M_m\times M_m$. Note that for each $n\in \omega$, $(A\cup C^0_n)\cap (B\cup C^1_n)\supset C^0_n$, which is positive. For any $i<j\in \omega$, we have that $(A\cup C^0_i)\cap (B\cup C^1_j)=\emptyset$. This is a contradiction as desired.
\end{proof}

For technical reasons, we need to extend the notion of compatibility to finite subsets of a poset.

\begin{definition}\label{definition: arity}
Let $P$ be a poset, $a,b\in [P]^{<\omega}$. Let 
\begin{enumerate}
\item $a || b$ denote that there are $r\in a$ and $t\in b$ such that $r || t$,
\item $a \perp b$ denote that $\neg (a || b)$.
\end{enumerate}
\end{definition}


\begin{definition}\label{definition: Knaster}
Fix $n\in \omega, k\leq \omega$, a poset $P$ and an ideal $J$ on $\kappa$. We say 
\begin{enumerate}
\item $[P]^{<k}$ is $J^n$-$m$-linked if for any $X\in (J^n)^+$, and $\langle p_a\in [P]^{<k}: a\in X\rangle$, there exists $Y\subset X$ in $(J^n)^+$ such that for any $\{a_i: i<m\}\subset Y$ with $a_i\simeq a_j$ for any $i,j<m$, then $\{p_{a_i}: i<m\}$ has a lower bound in $P$; 
\item $[P]^{<k}$ is $J^n$-Knaster if it is $[P]^{<k}$ is $J^n$-$2$-linked.
\end{enumerate}
\end{definition}

\begin{remark}
$[P]^{<\omega}$ being $J^n$-$m$-linked is equivalent to $P$ being $J^n$-$m$-linked. The reason why we consider the shift to finite sets of conditions is the proof of Lemma \ref{lemma:delta}. We choose such exposition for a cleaner presentation.
\end{remark}


%


\begin{lemma}\label{lemma: finitesupup}
If $\bigcup_{k\in \omega} P_k = P$ witnesses that $P$ is $\sigma$-$w$-finite-c.c, then for any $k\in \omega$ and any $\langle (a_i, b_i): i\in \omega\rangle \subset [P_k]^{<\omega} \times [P_k]^{<\omega}$ with $a_i || b_i$ for all $i\in \omega$, there exist $l<j\in \omega$ such that $a_l || b_j$.
\end{lemma}
\begin{proof}
Suppose $\langle (a_i, b_i): i\in \omega\rangle \subset [P_k]^{<\omega} \times [P_k]^{<\omega}$ witnesses otherwise for the sake of contradiction.
For each $i\in \omega$, pick $p_i\in a_i$ and $q_i\in b_i$ such that $p_i || q_i$. Then $\langle (p_i, q_i): i\in \omega\rangle \subset P_k \times P_k$ contradicts with the fact that $P=\bigcup_{k\in \omega} P_k$ is $\sigma$-$w$-finite-c.c.
\end{proof}

\begin{remark}
For $k<k'\leq \omega$ and $n\in \omega$, $[P]^{<k'}$ being $J^n$-Knaster implies that $[P]^{<k}$ is $J^n$-Knaster.
\end{remark}

The main result of this section is that:
\begin{theorem}\label{theorem: main}
Let $J$ be a uniform normal $\nu$-saturated ideal on $\kappa$ for some $\nu<\kappa$ and $P$ is a poset.
If $P$ is $\sigma$-w-finite-c.c, then for any $n\in \omega$, $[P]^{<\omega}$ is $J^n$-Knaster.
\end{theorem}

The structure of the proof of Theorem \ref{theorem: main} follows the one of $J^{n+1}$-weak-$\Delta$-system lemma. For the sake of completeness, we spell out the details. We will point out the locations where we need extra care since the poset given is abstract.

\subsection{An abstract $\Delta$-system lemma}
The same proof as in Proposition \ref{proposition: prop1d} gives the following.
\begin{proposition}\label{proposition: JKnaster}
If $P$ is $\sigma$-finite-c.c and $J$ is a uniform normal ideal on a uncountable regular $\kappa$, then $[P]^{<\omega}$ is $J$-Knaster.
\end{proposition}

The following key lemma essentially says we can obtain abstractly certain consequences of the $\Delta$-system lemma. Let $P$ be a $\sigma$-$w$-finite-c.c and let $\bigcup_{k\in \omega} P_k$ be the witness. 

\begin{lemma}\label{lemma:delta}
For any $Y\subset X\in J^+$, $\langle p_{\beta}\in [P]^{<\omega}: \beta\in Y\rangle$, $m\in \omega$, there exist $J$-positive $Y'\subset Y$, $F\in [Y]^{<\omega}$ and $q=\bigcup_{\gamma\in F} p_\gamma$ such that for any $r\in [P_m]^{<\omega}$,
\begin{itemize}
\item if $r || q$ then $r || p_\beta$ for $J^*\restriction Y'$-many $\beta\in Y'$, 
\item if $r \perp q$ then $r \perp p_\beta$ for $J^*\restriction Y'$-many $\beta\in Y'$,
\item $q || p_\beta$ for all $\beta\in Y'$.
\end{itemize}

\end{lemma}

\begin{proof}
We may assume there is some $l\in \omega$ such that $p_\beta\in P_l$ for all $\beta\in Y$ and $m\geq l$. Notice that applying Proposition \ref{proposition: JKnaster} to shrink $Y$ if necessary, the third requirement is taken care of automatically. As a result, we focus on the first two requirements.
Suppose for the sake of contradiction that the conclusion does not hold. Given $Z\subset Y\in J^+$, let $Q_Z$ denote the set $\{p_\gamma: \gamma\in Z\}$.
Consider the following assertion
 $(*)$: for any $J$-positive $Y'\subset Y$ and any $q\in [Q_{Y}]^{<\omega}$, there exists $r\in [P_m]^{<\omega}$ such that
\begin{enumerate}
\item either
 $r || q$, and $\{\eta\in Y': r \perp p_\eta\}\in J^+$,
 \item or $r \perp q$ and $\{\eta\in Y': r || p_\eta\}\in J^+$. 
\end{enumerate} 

By the hypothesis, $(*)$ holds. For notational simplicity, for relevant $q$, $r$ and $Y'$, we say $(*)_{q,r, Y'}^i$ holds if $(i)$ as above holds for $q, r$ and $Y'$, for $i=1,2$. 

\begin{claim}\label{claim: hereditary}
For any $Y'\subset Y \in J^+$, if for any $q\in [Q_{Y'}]^{<\omega}$, there is no $r\in [P_m]^{<\omega}$ satisfying $(*)^1_{q,r, Y'}$, then for any $Y''\subset Y' \in J^+$ and any $q\in [Q_{Y'}]^{<\omega}$, there is no $r\in [P_m]^{<\omega}$ satisfying $(*)^1_{q,r, Y''}$.
\end{claim}

\begin{proof}[Proof of the Claim]
Otherwise, let $Y'', q\in [Q_{Y'}]^{<\omega}$ and $r\in [P_m]^{<\omega}$ exemplify the negation. Then we have $r || q$ and $\{\eta\in Y'': r \perp p_\eta\}\in J^+$, which implies $(*)^1_{q,r, Y'}$, which is contradiction.
\end{proof}

We recursively build $\langle Z_n\subset Y, r_n\in P_m, q_n, d_n\in \{yes,no\}: n\in \omega\rangle$ as follows: 
\begin{enumerate}
\item $Z_0=Y$. We ask if there is $q\in [Q_{Z_0}]^{<\omega}$ for which there is some $r\in [P_{m}]^{<\omega}$ such that $(*)^1_{q,r, Z_0}$ holds. If the answer is yes, then we define $q_0, r_0$ to be some $q,r$ satisfying the above and $d_0=yes$. Let $Z_1=\{\eta\in Z_0: r\perp p_\eta\}\in J^+$. If the answer is no, let $q_0=p_{\min Z_0}$ and let $r_0\in [P_m]^{<\omega}$ be such that $(*)^2_{q_0,r_0, Z_0}$ holds and $d_0=no$. Let $Z_1=\{\eta\in Z_0: r_0 || p_\eta\}\in J^+$.
\item At stage $n+1$, suppose we already define $\langle Z_j, r_j, p_{\gamma_j}, d_j: j\leq n\rangle$ as well as $Z_{n+1}$. 
First let us assume that there is $j\leq n$ such that $d_j=no$. By Claim \ref{claim: hereditary}, for all $j'\in [j, n]$, $d_{j'}=no$. Let $q_{n+1}=p_{\min Z_{n+1}}$.
By Claim \ref{claim: hereditary} and $(*)$, we know that if we let $q^*=\bigcup_{j\leq v \leq n+1} q_v$, then there is some $r_{n+1}\in [P_m]^{<\omega}$ such that $(*)^2_{q^*, r_{n+1},Z_{n+1}}$ holds. Let $d_{n+1}=no$ and $Z_{n+2}=\{\eta\in Z_{n+1}: r_{n+1} || p_{\eta} \}\in J^{+}$. 
Suppose for all $j\leq n$, $d_j=yes$. In this case, we basically repeat the construction at stage $0$. More precisley, we ask if there is $q\in [Q_{Z_{n+1}}]^{<\omega}$ for which there is some $r\in [P_{m}]^{<\omega}$ such that $(*)^1_{q,r, Z_{n+1}}$ holds. If the answer is yes, then we define $q_{n+1}, r_{n+1}$ to be some $q,r$ satisfying the above and $d_{n+1}=yes$. Let $Z_{n+2}=\{\eta\in Z_{n+1}: r\perp p_\eta\}\in J^+$. If the answer is no, let $q_{n+1}=p_{\min Z_{n+1}}$ and let $r_{n+1}\in [P_m]^{<\omega}$ be such that $(*)^2_{q_{n+1},r_{n+1}, Z_{n+1}}$ holds and $d_{n+1}=no$. Let $Z_{n+2}=\{\eta\in Z_{n+1}: r || p_\eta\}\in J^+$.
\end{enumerate}

\begin{claim}\label{claim: finiteyes}
There is some $j<\omega$ such that $d_j=no$.
\end{claim}
\begin{proof}
Otherwise, $d_j=yes$ for all $j\in \omega$. By the construction, we have a sequence $\langle (r_n, q_n)\in [P_m]^{<\omega}\times [P_m]^{<\omega}: n \in \omega\rangle$ satisfying that for any $n<w<\omega$, $r_n || q_n$, and $r_n \perp q_w$. This contradicts with Lemma \ref{lemma: finitesupup}.
\end{proof}
By Claim \ref{claim: hereditary} and \ref{claim: finiteyes}, there exists a least $j\in \omega$ such that for all $j'\geq j$, $d_{j'}=no$. By the construction, we have a sequence $\langle (q_{n+1}, r_n)\in [P_m]^{<\omega}\times [P_m]^{<\omega}: j\leq n \in \omega\rangle$ satisfying that for any $l\leq n<w<\omega$, $q_{n+1} \perp r_w$, and $q_{n+1} || r_n$. Again, this contradicts with Lemma \ref{lemma: finitesupup}.\end{proof}

\begin{remark}
Lemma \ref{lemma:delta} is the only place where we actually use the stronger $\sigma$-$w$-finite-c.c-ness. The rest of the proof goes through only assuming the poset is $\sigma$-finite-c.c.
\end{remark}

\subsection{$J^{n+1}$-Knasterness}

Fix $J$, a uniform normal $\nu$-saturated ideal on $\kappa$ for some $\nu<\kappa$. If $n=1$, then it is enough that $J$ is precipitous \cite{GalvinJechMagidor}. However, we shall not optimize with this respect.

%
%

We will prove the following theorem in this section, by induction on $n\geq 1$. We repeat Theorem \ref{theorem: n+1-D} for the convenience of the reader. Actually, we prove a slightly stronger version below (replacing $P$ with $[P]^{<\omega}$).

\begin{theorem*}
Suppose $J$ is a uniform normal $\nu$-saturated ideal on $\kappa$ for some $\nu<\kappa$ and $n\in \omega$.
If $P$ is $\sigma$-$w$-finite-c.c, then $[P]^{<\omega}$ is $J^{n+1}$-Knaster.
\end{theorem*}

Fix $\bigcup_{n\in \omega} P_n = P$ witnessing the $\sigma$-$w$-finite-c.c-ness of $P$. Let $X\in (J^{n+1})^+$ and $\bar{p}=\langle p_a\in [P]^{<\omega}: a\in X\rangle$ be given. By shrinking $X$ if necessary and the $\sigma$-completeness of $J^{n+1}$, we may assume there is $k,l\in \omega$ such that for any $a\in X$, $p_a\in [P]^k$ and $p_a\subset P_l$. 
 
\subsubsection{Preliminary pruning}

For $i\leq n+1$, define $X\restriction i = \{a\in [\kappa]^i: \exists c\in [\kappa]^{n+1-i} (a,c)\in X\}$. It is clear that for all $i\leq n+1$, $X\restriction i \in (J^i)^+$. For each $a\in X\restriction i$ and $j> i$, let $(X\restriction j)_a = \{c\in [\kappa]^{j-i}: (a,c)\in X\restriction j\}$.

\begin{lemma}\label{lemma: roots}
There exists a $J^{n+1}$-positive $Y\subset X$ and $\langle q_{a}\in [P_l]^{<\omega}: a\in Y\restriction i, i\leq n+1\rangle$ such that 
\begin{enumerate}
	\item if $a\in Y$, then $q_a = p_a$,
	\item for any $i\leq n$, $a\in Y\restriction i$, and $r\in [P_l]^{<\omega}$, if 
		\begin{itemize}
		\item $r || q_a$, then for $J^*\restriction (Y\restriction i+1)_a$-many $\gamma$, $r || q_{a^\frown \gamma}$,
		\item $r \perp q_a$, then for $J^*\restriction (Y\restriction i+1)_a$-many $\gamma$, $r \perp q_{a^\frown \gamma}$,
		\item for any $\gamma\in (Y\restriction i+1)_a$, $q_a || q_{a^\frown \gamma}$.
		\end{itemize}
\end{enumerate}
\end{lemma}

\begin{proof}
We induct on $n\geq 0$. When $n=0$, Lemma \ref{lemma:delta} gives the desired conclusion. Suppose $n>0$. For each $a\in X\restriction n$, apply Lemma \ref{lemma:delta} to get $q_a\in [P_l]^{<\omega}$ and $J$-positive $Z_a\subset (X)_a$ such that 
	\begin{itemize}
	\item if $r\in [P_l]^{<\omega}$ and $r || q_a$, then for $J^*\restriction Z_a$-many $\gamma\in Z_a$, $r|| q_{a^\frown \gamma}$,
	\item if $r\in [P_l]^{<\omega}$ and $r \perp q_a$, then for $J^*\restriction Z_a$-many $\gamma\in Z_a$, $r\perp q_{a^\frown \gamma}$, and 
		\item $q_a || q_{a^\frown \gamma}$ for all $\gamma\in Z_a$.
	\end{itemize}
Apply the induction hypothesis to $X\restriction n$ and $\langle q_a: a\in X\restriction n\rangle$. We get a $J^n$-positive $Y^*\subset X\restriction n$ and $\langle q_a: a\in Y^*\restriction i, i\leq n\rangle$ satisfying the conclusion of the lemma. Finally, it is not hard to see that $Y=\{(a,\gamma): a\in Y^*, \gamma\in Z_a\}$ and $\langle q_a: a\in Y^*\restriction i, i\leq n\rangle \cup \langle q_a=_{\mathrm{def}}p_a: a\in Y\rangle$ are as desired.
\end{proof}

\subsubsection{Non-overlapping increasing types}
\begin{lemma}\label{lemma:n+1disjoint}
There exists a $J^{n+1}$-positive $Y\subset X$ such that for any $a<b\in Y$, $p_a || p_b$.
\end{lemma}

\begin{proof}
Given $a\in X$ and $(J^{n+1})^+$-positive $Z\subset X$, we say $a$ is \emph{good} for $Z$ if $\{b\in Z: p_a || p_b\} \in (J^{n+1})^*\restriction Z$.

\begin{claim}\label{claim: start}
There exists $Y_0\subset X$ in $(J^{n+1})^+$ such that for any $a\in Y_0$, $a$ is good for $Y_0$.
\end{claim}
\begin{proof}[Proof of the Claim]
Suppose otherwise for the sake of contradiction. Using the assumption, recursively build $\langle X_n: n\in \omega\rangle$ and $\langle a_n: n\in \omega\rangle$ such that 
	\begin{enumerate}
	\item $X_0=X$, $X_{n+1}\subset X_n$ and $X_n\in J^+$,
	\item $a_n\in X_n$ is not good for $X_n$ and
	\item for any $b\in X_{n+1}$, $p_{a_n}\perp p_b$, in particular, for any $n<m\in \omega$, $p_{a_n}\perp p_{a_m}$.
	\end{enumerate}
	The construction is exactly the same as the one in Proposition \ref{proposition: JKnaster} and we can arrive at the contradiction the same way.
\end{proof}

Fix $Y_0$ as in Claim \ref{claim: start}. For each $a\in Y_0$, let $A_a=\{b\in Y_0: p_a || p_b\}\in (J^{n+1})^*\restriction Y_0$. Applying Lemma \ref{lemma: highernormality}, it is easy to see $Y=_{\mathrm{def}} Y_0\cap \Delta_{a\in Y_0} A_a$ is as desired.
\end{proof}

\subsubsection{Overlapping at the last coordinate}

\begin{lemma}\label{lemma:lastoverlap}
There exists a $J^{n+1}$-positive $Y\subset X$ such that for any $a\simeq b\in Y$ with $a_n=b_n$, $p_a || p_b$.
\end{lemma}

\begin{proof}
Let $G\subset P(\kappa)/J$ be generic over $V$. Then in $V[G]$, fix a generic elementary embedding $j: V\to M$ with $\mathrm{crit}(j)=\kappa$ and $V[G]\models {}^\kappa M\subset M$. 
We can further assume $V[G]$ satisfies the following $B=\{a\in X\restriction n: \kappa\in j((X)_a)\} \in (\bar{J}^n)^+$. The reason is that if $\Vdash_{P(\kappa)/J} B\in \bar{J}^n$, then there exists $B_0\in J^n$ such that $\Vdash_{P(\kappa)/J} B\subset B_0$ by Fact \ref{fact: standard}. Let $a\in X\restriction n\cap (B_0)^c$, and force below $(X)_a$ over $P(\kappa)/J$ to get $V[G]$. Then in $V[G]$, $a\in B \cap (B_0)^c$, contradicting with the fact that $B\subset B_0$.

Consider $D=\{j(\bar{p})_{a^\frown \kappa}\in j(P): a\in B\}$. Since $M\models j(P)$ is $\sigma$-$w$-finite-c.c, $V[G]\models j(P)$ is $\sigma$-$w$-finite-c.c since well-founded-ness is absolute between transitive models (or alternatively use the fact that $V[G]\models {}^\omega M\subset M$).

By Fact \ref{fact: standard}, $\bar{J}$ is $\nu$-c.c in $V[G]$. Apply the induction hypothesis in $V[G]$ to $D$, $B$ and $\bar{J}$, we get $B^*\subset B$ which is $\bar{J}^n$-positive such that for any $c,d\in B^*$, if $c\simeq d$, then $j(\bar{p})_{c^\frown \kappa} || j(\bar{p})_{d^\frown \kappa}$.

Back in $V$. For each $a\in X\restriction n$, find $Z_a\subset (X)_a\in J^+$ forcing over $P(\kappa)/J$ that $a\in \dot{B}^*$ if such a set exists. 

	\begin{claim}
	$Y^*=\{a\in X\restriction n: Z_a\text{ exists}\}\in (J^n)^+$.
	\end{claim}
	\begin{proof}[Proof of the Claim]
	Suppose otherwise for the sake of contradiction. There exists $C\in (J^n)^*\restriction (X\restriction n)$ such that for any $a\in C$, $(X)_a\Vdash a\not\in \dot{B}^*$. In $V[G]$, since $B^*\subset X\restriction n$ is in $(\bar{J}^n)^+$, we know that $B^*\cap C\neq \emptyset$. Fix $d\in B^*\cap C$. In particular, since $d\in B$, $\kappa\in j((X)_d)$, which implies that $(X)_d\in G$. As $(X)_d\Vdash d\not\in \dot{B}^*$, we have $d\not\in B^*$. This is a contradiction.
	\end{proof}
	
	For any $a\simeq b\in Y^*$, if $Z_a\cap Z_b\in J^+$, then there is $C_{a,b}\in J^*$ such that for any $\gamma\in C_{a,b}\cap Z_a\cap Z_b$, $p_{a^\frown \gamma} || p_{b^\frown \gamma}$. The reason is that $Z_a\cap Z_b \Vdash j(\bar{p})_{a^\frown \kappa} || j(\bar{p})_{b^\frown \kappa}$. For any $a\simeq b\in Y^*$ with $Z_a\cap Z_b\in J$, let $C_{a,b}\in J^*$ disjoint from $Z_a\cap Z_b$.
	
	For each $b\in Y^*$, let $$Z_b'=\{\gamma\in Z_b: \forall a<\gamma, a\in Y^*, a\simeq b, \gamma\in C_{a,b}\}.$$
	
	Note that $Z_b'=_J Z_b$ since $J$ is normal. Finally, let $Y=\{(b,\gamma): b\in Y^*, \gamma\in Z_b'\}$. It is easy to see that $Y\subset X$ is in $(J^{n+1})^+$. To see that $Y$ is as desired, fix $a\simeq b\in Y$ with $a_n=b_n=\gamma$. By the definition of $Z_b'$ and $Z_a'$, $\gamma\in C_{a,b}\cap Z_a\cap Z_b$. In particular, $Z_a\cap Z_b\in J^+$. As a result of the choice of $C_{a,b}$ in this case, $p_{a^\frown \gamma} || p_{b^\frown \gamma}$.	\end{proof}

\subsubsection{Other types}

We classify the remaining types as follows. 

\begin{definition}\label{definition: typek}
We say $(a,b)\in [X]^2_{\simeq} =\{(c,d)\in X\times X: c_n\leq d_n, c\simeq d\}$ is of \emph{type $k$} for $0<k\leq n$ if $|b-a_n|=k$.
\end{definition}

\begin{remark}
The reason why we focus on $k\in (0,n+1)$ is that $(a,b)$ of type $0$ implies $a_n=b_n$ and $(a,b)$ of type $n+1$ implies $a<b$. Both of these types have been dealt with previously.
\end{remark}

\begin{proposition}\label{proposition: typek}
Fix $k\in (0,n+1)$. There exists a $J^{n+1}$-positive $Y\subset X$ such that for any $(a,b)\in [Y]^2_{\simeq}$ of type $k$, it is true that $p_a || p_b$.
\end{proposition}

\begin{proof}
By shrinking $X$ if necessary, we may without loss of generality assume that there are $\langle q_a: a\in X\restriction i, i\leq n+1\rangle$ satisfying the conclusion of Lemma \ref{lemma: roots}. Apply the induction hypothesis to $\langle q_a: a\in X\restriction n\rangle$ to get $J^n$-positive $Z\subset X\restriction n$ such that for any $c,d\in Z$ with $c\simeq d$, $q_c || q_d$. 
We say $(a^0, {b^0})$ is \emph{potentially of type $k$} if 
	\begin{itemize}
	\item $a^0\in Z$,
	\item ${b^0}\in Z\restriction n-k$,
	\item there are ${a^1}, {b^1}$ such that $({a^0}^\frown {a^1}, {b^0}^\frown {b^1})\in [X]^2_{\simeq}$ of type $k$.
	\end{itemize}

	\begin{claim}\label{claim: assym}
	For any $(a^0,{b^0})$ potentially of type $k$, there are $C_{a^0}=C_{a^0,{b^0}}^0\in J^*$ and $C_{{b^0}}=C^1_{a^0,{b^0}}\in (J^k)^*$ such that for any ${a^1}\in C_{a^0}\cap (X)_{a^0}$ and ${b^1}\in C_{{b^0}}\cap (X)_{{b^0}}$ with ${a^1}<{b^1}$, $p_{{a^0}^\frown {a^1}} || p_{{b^0}^\frown {b^1}}$.
	\end{claim}
	\begin{proof}[Proof of the claim]
	Let us first show that $q_{a^0} || q_{{b^0}}$. Suppose otherwise for the sake of contradiction. By the conclusion of Lemma \ref{lemma: roots} and the fact that $(Z)_{{b^0}}$ is a $J^{k-1}$-positive subset of $(X\restriction n)_{{b^0}}$, there exists $c \in (Z)_{{b^0}}$ such that $a^0\simeq {b^0}^\frown c$ and $q_{a^0}\perp q_{{b^0}^\frown c}$. This contradicts with the choice of $Z$. Therefore, $q_{a^0} || q_{{b^0}}$. 
	
	Applying Lemma \ref{lemma: roots} again with $q_{{b^0}}$ playing the role of $r$ as in Lemma \ref{lemma: roots}, we can find $C_{a^0}\in J^*$ such that $q_{{b^0}} || p_{{a^0}^\frown {a^1}}$ for all ${a^1}\in C_{a^0}\cap (X)_{a^0}$. For each ${a^1}\in C_{a^0}$, applying Lemma \ref{lemma: roots} once again with $p_{{a^0}^\frown {a^1}}$ playing the role of $r$, we can find $D_{{b^0},{a^1}}\in (J^k)^*$ such that for any ${b^1}\in D_{{b^0},{a^1}}\cap (X)_{{b^0}}$, $p_{{a^0}^\frown {a^1}} || p_{{b^0}^\frown {b^1}}$. Finally, by Lemma \ref{lemma: highernormality}, let $C_{{b^0}}=_{def} \Delta_{{a^1}\in C_{a^0}} D_{{b^0},{a^1}}=_{def}\{{b^1}: \forall {a^1}\in C_{a^0}\cap \min {b^1}, {b^1}\in D_{{b^0}, {a^1}}\} \in (J^k)^* $. It is easy to see that $C_{a^0}$ and $C_{{b^0}}$ are as desired.
	\end{proof}
	
For each $b\in Z\restriction n-k$, let $$W_b = \{{b^1}: \forall a^0\subset \min {b^1}, (a^0,b) \text{ is potentially of type $k$, }{b^1}\in C^1_{a^0, b} \}.$$
By Lemma \ref{lemma: highernormality}, $W_b\in (J^k)^*$. Let $$Y_0=\{(c,d): c\in Z\restriction n-k, d\in W_c\cap (Z)_c\}.$$
Then $Y_0$ is a $J^{n+1}$-positive subset of $X$.
For each $a\in Y_0\restriction n$, let
$$E_a=\{\gamma: \forall {b^0}\subset \gamma, (a,{b^0}) \text{ is potentially of type $k$, }\gamma\in C^0_{a,{b^0}}\}.$$
By Lemma \ref{lemma: highernormality}, $E_a\in J^*$. Let $Y=\{(a,\gamma)\in Y_0: \gamma\in E_a\cap (Y_0)_a\}$. It is immediate that $Y$ is a $J^{n+1}$-positive subset of $Y_0$.

Finally, let us verify that $Y$ is as desired. Suppose $(a,b)\in [Y]^2_{\simeq}$ of type $k$. Let $a^0=a\restriction n$ and ${b^0}=b\restriction n-k$. Then $(a^0, {b^0})$ is potentially of type $k$. Since ${b^0}<a_n$, $a_n\in C^{0}_{a^0, {b^0}}$ and since $a^0< b_{n-k+1}$, $b\restriction_{[n-k+1, n]}\in C^1_{a^0, {b^0}}$. Since $a_n < b\restriction_{[n-k+1, n]}$, by Claim \ref{claim: assym}, $p_{a} || p_{b}$.
\end{proof}
By going through all the finitely many types, we get the following: 
\begin{corollary}\label{cor: alltypes}
There is a $J^{n+1}$-positive $Y\subset X$ such that for any $k\in (0, n+1)$ and $(a,b)\in [Y]^2_{\simeq}$ of type $k$, it is true that $p_a || p_b$.
\end{corollary}

\begin{proof}[Proof of Theorem \ref{theorem: main}]
Apply successively Lemma \ref{lemma:n+1disjoint}, Lemma \ref{lemma:lastoverlap} and Corollary \ref{cor: alltypes}.
\end{proof}

\section{Concluding remarks}\label{section: openquestions}

We do not know if the ``$\sigma$-$w$-finite-chain condition" strengthening of ``$\sigma$-finite-chain condition" is necessary for getting our results.

\begin{question}
If $J$ is a uniform normal ideal on $\kappa$ such that $P(\kappa)/J$ satisfies the $\sigma$-finite-chain condition, is it true that $P(\kappa)/J$ is $J^n$-Knaster for all $n\in \omega$?
\end{question}

More generally, one can define the following ``linked" strengthening of the $\sigma$-finite-chain condition. 

\begin{definition}
A poset $P$ is said to satisfy \emph{$\sigma$-finite-$m$-chain condition} where $m\in \omega$, if there exists $\langle P_n: n\in \omega\rangle$ such that 
	\begin{enumerate}
	\item $P=\bigcup_{n\in \omega} P_n$, and 
	\item for each $n\in \omega$, for any infinite $\langle p_i\in P_n: i\in \omega\rangle$, there exists $H\in [\omega]^m$ such that $\{p_i: i\in H\}$ admits a lower bound in $P$.
	\end{enumerate}
\end{definition}

\begin{question}
If $J$ is a uniform normal ideal on $\kappa$ and $m\in \omega$ such that $P(\kappa)/J$ satisfies the $\sigma$-finite-$m$-chain condition, is it true that $P(\kappa)/J$ is $J^n$-$m$-linked for all $n\in \omega$?
\end{question}

%

\section{Acknowledgment}
We thank the anonymous referee for their helpful comments, suggestions and corrections that greatly improve the paper.

\bibliographystyle{amsplain}
\bibliography{bib}

\end{document}